\documentclass[a4paper,twoside,reqno]{amsart}
\usepackage{amssymb}
\usepackage[mathscr]{eucal}

\usepackage{graphicx}
\usepackage{xcolor}

\input xy
\xyoption{all}

\theoremstyle{plain}
\newtheorem{prop}{Proposition}[section]
\newtheorem{thm}[prop]{Theorem}
\newtheorem{cor}[prop]{Corollary}
\newtheorem{lem}[prop]{Lemma}

\theoremstyle{definition}
\newtheorem{dfn}[prop]{Definition}

\newtheorem{lab}[prop]{}

\theoremstyle{definition}
\newtheorem{example}[prop]{Example}
\newtheorem{Examples}[prop]{Examples}
\newtheorem{rem}[prop]{Remark}
\newtheorem{Rems}[prop]{Remarks}

\newcounter{reml}

\newenvironment{remlist}{\begingroup\setcounter{reml}{0}}
  {\endgroup}

\newcommand{\invisiblechar}{\rule{1em}{0ex}}
\newcommand{\emptyline}{\invisiblechar\par}

\newcommand{\into}{\hookrightarrow}
\newcommand{\onto}{\twoheadrightarrow}
\renewcommand{\iff}{\Leftrightarrow}

\newcommand{\A}{{\mathbb{A}}}

\newcommand{\G}{{\mathbb{G}}}
\newcommand{\N}{{\mathbb{N}}}
\renewcommand{\P}{{\mathbb{P}}}

\newcommand{\R}{{\mathbb{R}}}
\newcommand{\Z}{{\mathbb{Z}}}

\newcommand{\scrP}{\mathscr{P}}
\newcommand{\scrS}{\mathscr{S}}

\newcommand{\scrX}{\mathscr{X}}

\renewcommand{\epsilon}{\varepsilon}

\DeclareMathOperator{\Hom}{Hom}
\DeclareMathOperator{\im}{im}

\DeclareMathOperator{\Sper}{Sper}
\DeclareMathOperator{\supp}{supp}

\DeclareTextFontCommand{\textnf}{\normalfont}

\newcommand{\id}{\mathrm{id}}

\newcommand{\rey}{\rho}

\newcommand{\Sat}{\mathrm{Sat}}

\newcommand{\comp}{\mathbin{\scriptstyle\circ}} 
\newcommand{\du}{{\scriptscriptstyle\vee}}
\renewcommand{\emptyset}{\varnothing}

\newcommand{\hvar}{{-}} 

\newcommand{\ol}{\overline}
\newcommand{\plus}{{\scriptscriptstyle+}}

\newcommand{\To}{\Rightarrow}

\newcommand{\wt}[1]{\widetilde{#1}}

\newcommand{\qu}{/\!/\,}

\newcommand{\sa}{semi-algebraic}

\newcommand{\Label}[1]{\label{#1}}

\newcommand{\oldversion}[1]{}

\begin{document}

\title
[Sums of squares and moment problems in equivariant situations]
{Sums of squares and moment problems\\[3pt]in equivariant situations}


\author{Jaka Cimpri\v c}
   \address{Faculty of Mathematics and Physics,
   University of Ljubljana,
   Jadranska 19, SI-1000 Ljubljana, Slovenija}
   \email{cimpric@fmf.uni-lj.si}
   \urladdr{http://www.fmf.uni-lj.si/\textasciitilde cimpric}

\author{Salma Kuhlmann}
   \address{Department of Mathematics and Statistics,
University of Saskatchewan,
Room 142 McLean Hall,
106 Wiggins Road,
Saskatoon, SK S7N 5E6,
Canada}
   \email{skuhlman@snoopy.usask.ca }
   \urladdr{http://math.usask.ca/\textasciitilde skuhlman}

\author{Claus Scheiderer}
   \thanks{The third author was partially supported by the European
     RTNetwork RAAG, HPRN-CT-2001-00271}
   \address{Fachbereich Mathematik und Statistik\\
     Universit\"at Konstanz\\
     D-78457 Konstanz\\
     Germany}
   \email{claus.scheiderer@uni-konstanz.de}
   \urladdr{http://www.math.uni-konstanz.de/\textasciitilde scheider}

\keywords{semialgebraic sets, group actions, moment problems}
\subjclass[2000]{14P10,14L30,20G20}

\begin{abstract}
We begin a systematic study of positivity and moment problems in an
equivariant setting. Given a reductive group $G$ over $\R$ acting on an
affine $\R$-variety $V$, we consider the induced dual action on the
coordinate ring $\R[V]$ and on the linear dual space of $\R[V]$. In
this setting, given an invariant closed semialgebraic subset $K$ of
$V(\R)$, we study the problem of representation of invariant nonnegative
polynomials on $K$ by invariant sums of squares, and the closely related
problem of representation of invariant linear functionals on $\R[V]$ by
invariant measures supported on $K$. To this end, we analyse the relation
between quadratic modules of $\R[V]$ and associated quadratic modules of
the (finitely generated) subring $\R[V]^G$ of invariant polynomials. We
apply our results to investigate the finite solvability of an equivariant
version of the multidimensional $K$-moment problem. Most of our results
are specific to the case where the group $G(\R)$ is compact. 
\end{abstract}

\date\today
\maketitle


\section*{Introduction}

The study of positivity versus sums of squares has a long and
illustrious tradition, starting with Minkowski and Hilbert. It is
remarkable how the interest in such questions has risen in the last
decade, in particular in fields outside real algebraic geometry. The
application of semidefinite programming methods is currently turning
sums of squares into an efficient tool in polynomial optimization.
Another field to which sums of squares methods have recently been
applied successfully is moment problems and related questions in
analysis. Yet another, and still fresh, development is the study of
sums of squares in non-commutative settings, and the exploration
of its use in applications like engineering. An excellent overview of
such ongoing developments can be found in \cite{HP}.

Both in theoretical considerations and in practical applications, it
happens quite often that the situation in question allows symmetries.
For example, consider the problem of characterizing all polynomials
which are non-negative on a given closed subset $K$ of $\R^n$. If $K$
is invariant under some subgroup $G$ of the general linear group, one
may ask for a characterization of the $G$-invariant polynomials which
are non-negative on $K$. How can they be described in terms of
invariant sums of squares, or even, in terms of sums of squares of
invariants?

The first to systematically study equivariant situations in real
algebraic geometry were Procesi and Schwarz, in their landmark paper
\cite{PS}. Their main result characterizes the \emph{real} orbit
space of a matrix group $G$, in the case when $G$ is compact. In
particular, they showed that the real orbit space is described by
(finitely many) simultaneous inequalities which are essentially
explicit. These results were later refined by Br\"ocker \cite{Br}
who studied, in particular, the minimum number of inequalities needed
for a description. A more recent contribution is the paper \cite{GP}
by Gatermann and Parrilo, who were primarily interested in the
effective aspects of sums of squares decompositions of invariant
polynomials. We will review some of the main results of all three
papers later in more detail.

Apart from these works, positivity and sums of squares questions have
hardly been considered in equivariant settings yet, as far as we
know. It is one purpose of our paper to begin a systematic study. For
this reason, at least some of it is of foundational nature.

Generally, we have tried working over an arbitrary real closed base
field $R$ as much as possible, instead of only over the classical
real numbers $\R$. One reason is that the validity of a result over
arbitrary $R$ often has implications for the case of base field $\R$,
e.g.\ for the existence of complexity bounds, as is well known.

So our setup consists of a reductive linear group $G$ acting on an
affine variety $V$, the action being defined over a real closed field
$R$. Most of our results are specific to the case where the group
$G(R)$ is \sa ally compact, that is, closed and bounded in some $GL_N
(R)$.

After recalling and introducing a few technical notions from real
algebra (Sect.~1), we review the main results of Procesi-Schwarz,
Br\"ocker and Gatermann-Parrilo in Sect.~2. We generalize them to
arbitrary real closed base fields and give alternative, purely
algebraic proofs for some parts. The main result of \cite{GP} is
generalized to the case of a compact (infinite) group of symmetries.
Both our formulation and our proof are more conceptual and less
matrix-based than the original account in \cite{GP}. Sect.~3 contains
basic material. We consider an affine $G$-variety $V$ as above, with
$G(R)$ \sa ally compact. We study the operations of contraction and
extension, between $R[V]$ and $R[V]^G$, on the cones of sums of
squares, and on more general quadratic modules, and also the effect
of the Reynolds operator $\rey$. Also, we study the relation between
quadratic modules and their associated closed sets, both in $V$ and
$V\qu G$. A central result proved here is that $\rey$ maps the cone
$\Sigma R[V]^2$ of sums of squares into itself, and hence $\rey\bigl(
\Sigma R[V]^2\bigr)=\bigl(\Sigma R[V]^2\bigr)^G$, the cone of
$G$-invariant sums of squares. In fact, this characterizes the case
$G(R)$ \sa ally compact, for in Sect.~4 we prove that in all other
cases (with $G$ reductive), there exists $f\in R[G]$ with $\rey(f^2)
<0$. (Here $\rey$ is the Reynolds operator of $G$ acting on itself by
translation). Back to the compact case, we show by examples in
Sect.~5 that $(\Sigma R[V]^2)^G$ usually fails to be finitely
generated, as a preordering in $R[V]^G$. Other instructive examples
are presented in this section as well.

We would like to emphasize the importance of the Reynolds operator
$\rey$ (projection to the invariants), which plays a key role
throughout. Its consequent use replaces the tool of invariant
integration, which is available for $R=\R$ but not otherwise. Note
that the Reynolds operator is well understood from a computational
view point, not only for finite $G$ (see \cite{DK}).

The second part of the paper deals with moment problems having
symmetries. Of course, the base field is now $\R$, the usual real
numbers. For simplicity, let us assume that $K$ is a (basic)
closed subset of $\R^n$. The $K$-moment problem asks for a
characterization of all $K$-moment functionals on $\R[x]=\R[x_1,
\dots,x_n]$, that is, of all linear functionals $L\colon\R[x]\to\R$
arising from integration by some Borel measure on $K$ (all of whose
moments exist). Suppose that $K$ is invariant under the linear group
$G$. Then we may ask for a characterization of the
\emph{$G$-invariant} $K$-moment functionals. This task should
generally be easier, compared to a characterization of all $K$-moment
functionals. We isolate two conditions, the invariant resp.\ the
averaged moment property, (IMP) resp.\ (AMP), under which a
characterization of the $G$-invariant moment functionals is possible
in a reasonable (finitistic) sense. We give various equivalent
characterizations of these conditions. For example, (IMP) is
equivalent to the usual moment property for the image of $K$ in the
orbit space. By means of several examples we demonstrate that indeed
we gain something by considering these properties: (IMP) resp.\ (AMP)
are shown to hold in cases where the usual moment problem fails to be
finitely solvable. On the other hand, we prove an equivariant version
of the main result of \cite{PSch}. A large class of cases was
characterized there by a geometric condition, in which the moment
problem is not finitely solvable; we show that under a similar
equivariant geometric condition, not even the invariant moment problem
is finitely solvable.

The paper closes with a list of open questions. In a previous version of
this paper we asked whether (SMP) implies (IMP) and whether (AMP)
implies (IMP).  Recently, both questions were answered to the negative
by Tim Netzer (see example \ref{timbsp}. We would like to take the
opportunity to thank him for allowing us to include his example in our
paper.


\section{Notation and preliminaries}

All rings are commutative and have a unit. For general background on
real algebra and geometry we refer to \cite{BCR}, \cite{KS} and
\cite{PD}. In particular, background on the notion of the real
spectrum may be found in these texts. We will always denote the real
spectrum of a ring $A$ by $\Sper A$.

\begin{lab}\Label{qmpo}
Let $A$ be a ring. A \emph{quadratic module} in $A$ is a subset $M$
of $A$ with $1\in M$ which satisfies $M+M\subset M$ and $a^2M\subset
M$ for every $a\in A$.
A \emph{preordering} in $A$ is a quadratic module $T$ which is closed
under multiplication: $TT\subset T$. The smallest quadratic module in
$A$ (which happens to be a preordering) is $\Sigma A^2$, the set of
all sums of squares in $A$.

The quadratic module generated by $a_1,\dots,a_r\in A$ in $A$
consists of all elements
$$f=s_0+s_1a_1+\cdots+s_ra_r$$
with $s_i\in\Sigma A^2$, and is denoted by $QM_A(a_1,\dots,a_r)$. The
preordering generated by $a_1,\dots,a_r$ in $A$ is the quadratic
module generated by the $2^r$ products $a_1^{e_1}\cdots a_r^{e_r}$,
$e_i\in\{0,1\}$; it is denoted by $PO_A(a_1,\dots,a_r)$. A quadratic
module (resp., preordering) is said to be \emph{finitely generated}
if it can be generated by finitely many elements.
\end{lab}

\begin{lab}
In the first part of our paper we will work over general real closed
base fields $R$. By an affine $R$-variety we always mean an affine
algebraic variety $V$ defined over $R$ which is reduced. Its
coordinate ring is denoted $R[V]$, this is a finitely generated
reduced $R$-algebra. The set of $R$-points on $V$ is denoted $V(R)$.
We will use the notion of \sa\ subsets of $V(R)$. Recall the operator
tilda (see any of the above references): For $S$ a \sa\ subset of
$V(R)$, $\wt S$ is a constructible subset of $\Sper R[V]$, and
$S\mapsto\wt S$ is a bijection between \sa\ sets in $V(R)$ and
constructible sets in $\Sper R[V]$ which preserves all boolean
operations. In particular, $\Sper R[V]=\wt{V(R)}$.
\end{lab}

\begin{lab}\Label{notat}
Let again $A$ be a ring. Associated with a quadratic module $M$ in
$A$ (or in fact with any subset $M$ of $A$) is the closed subset
$$\scrX(M):=\scrX_A(M):=\{\alpha\in\Sper A\colon f(\alpha)\ge0
\text{ for every }f\in M\}$$
of $\Sper A$. If one thinks of the elements of $\Sper A$ as prime
cones in $A$, then $\scrX(M)$ is the set of prime cones which contain
$M$.
If $M$ is finitely generated, say by $a_1,\dots,a_r$, then $\scrX(M)=
\{\alpha\colon a_1(\alpha)\ge0,\dots,\alpha_r(\alpha)\ge0\}$ is a
constructible subset of $\Sper A$.

In the geometric situation we use the following notation. Let $R$ be
a real closed field and $V$ an affine $R$-variety. If $M$ is a
quadratic module in $R[V]$ we write
$$\scrS(M):=\scrS_V(M):=\bigcap_{f\in M}\{x\in V(R)\colon f(x)\ge0\}=
V(R)\cap\scrX_{R[V]}(M)$$
for the `trace' of $\scrX(M)\subset\Sper R[V]=\wt{V(R)}$ in $V(R)$.
This is a closed subset of $V(R)$, which is \sa\ if $M$ is finitely
generated.
\end{lab}

\begin{lab}\Label{dfnsaturated}
Let $M$ be a subset of $A$. The \emph{saturation} of $M$ is the
preordering
\begin{center}
$\Sat(M):=\Sat_A(M):=\{f\in A\colon f\ge0$ on $\scrX(M)\}$
\end{center}
of $A$. We have $M\subset\Sat(M)$ tautologically, and $M$ is called
\emph{saturated} if equality holds. With every subset $X$ of
$\Sper A$ we can associate a saturated preordering of $A$, namely
$$\scrP(X):=\scrP_A(X):=\{a\in A\colon a\ge0\text{ on }X\}.$$
The two operators $\scrP$ and $\scrX$ set up a Galois correspondence
between the subsets of $\Sper A$ and the subsets of $A$, the closed
objects of which are the pro-basic closed subsets of $\Sper A$ on one
side and the saturated preorders of $A$ on the other. In particular,
we have $\Sat(M)=\scrP(\scrX(M))$ for every quadratic module $M$.

See also \cite{Sch:guide} for a more detailed discussion of these
notions.
\end{lab}


\section{Actions of reductive groups on affine $R$-varieties}

\begin{lab}
Let always $R$ be a real closed field. By an $R$-variety $V$ we mean
a reduced separated scheme of finite type over $R$. Most
$R$-varieties in this paper will be affine. Affine $R$-varieties $V$
correspond, in a contravariant functorial way, to finitely generated
$R$-algebras which are reduced (without nilpotent elements $\ne0$);
namely, $V$ corresponds to its coordinate ring, $R[V]$. The set of
$R$-points on $V$ is denoted $V(R)$, as usual.

Always let $G$ be a linear algebraic group over $R$. The Zariski
closure of $G(R)$ in $G$ is an open and closed subgroup of $G$ (of
finite index), so it contains the identity component $G_0$ of $G$.
(This is in fact true for linear groups over any ground field, see
\cite{Bo} Cor.\ V.18.3.) Throughout this paper we will assume that
$G(R)$ is Zariski dense in $G$, or equivalently, that every connected
component of $G$ contains an $R$-point.
\end{lab}

\begin{lab}
Let $G$ act on the affine $R$-variety $V$ by means of a morphism
$$G\times V\to V,\quad(g,x)\mapsto gx$$
of $R$-varieties. To such an action corresponds the dual action
$$\eta=\eta_V\colon R[V]\to R[G]\otimes_RR[V]$$
which is a homomorphism of $R$-algebras. The group $G(R)$ acts on
the $R$-algebra $R[V]$ through algebra automorphisms on the right by
$(f,g)\mapsto f^g$, where $f^g$ is characterized by $f^g(x)=f(gx)$.
Thus, if $\eta_V(f)=\sum_ia_i\otimes f_i$, then
$$f^g=\sum_ia_i(g)\cdot f_i$$
($f\in R[V]$, $g\in G(R)$). Using the assumption that $G(R)$ is
Zariski dense in $G$, it is easy to see that an element $f\in R[V]$
is invariant under this action of $G(R)$ if and only if $\eta(f)=
1\otimes f$.
One writes
$$R[V]^G:=R[V]^{G(R)}=\{f\in R[V]\colon\eta(f)=1\otimes f\}.$$
This is the subring of \emph{$G$-invariants} of $R[V]$.

To simplify language, we will say that a subset of $V(R)$, or of
$R[V]$, is \emph{$G$-invariant}, if it is invariant under the action
of the group $G(R)$.
\end{lab}

\begin{lab}
Recall that a linear group $G$ over $R$ is called \emph{reductive} if
$G$ contains no non-trivial unipotent normal closed subgroup.
(Sometimes it is also required that $G$ is connected.) Since $R$ has
characteristic zero, it is equivalent that every finite-dimensional
$G$-module $M$ is a direct sum of irreducible $G$-modules, or
equivalently, that every $G$-submodule of such $M$ has a
$G$-invariant complement. (See, e.~g., \cite{DK}, Theorems 2.2.13 and
2.2.5.)

Most of our results will focus on the case where $G(R)$ is \sa ally
compact (i.e. $G(R)$ is  closed and bounded in an affine $R$-space).
Such a linear group $G$ over $R$ is always reductive. Indeed, $G(R)$
cannot contain any closed subgroup isomorphic to $R=\G_a(R)$ (the
additive group of $R$), and hence $G$ cannot contain any non-trivial
unipotent closed $R$-subgroup (normal or not).

When $G$ is reductive, then for any affine $G$-variety $V$ the ring
$R[V]^G$ of $G$-invariants is finitely generated as an $R$-algebra,
as was shown by Hilbert (see \cite{DK}, 2.2.10). Thus $R[V]^G$ is the
coordinate ring of an affine $R$-variety $W$, which is called the
\emph{quotient variety} of $V$ by $G$ and is commonly denoted
$V\qu G$:
$$W=V\qu G,\quad R[W]=R[V]^G.$$
Corresponding to the inclusion $R[W]=R[V]^G\subset R[V]$ of
$R$-algebras we have the morphism $\pi\colon V\to W$ of affine
$R$-varieties, called the \emph{quotient morphism}.
\end{lab}

\begin{lab}
As a morphism of $R$-varieties, $\pi\colon V\to W=V\qu G$ is
surjective and open. The induced map $\pi\colon V(R)\to W(R)$ on
$R$-points, however, fails to be surjective in general, and it is an
important problem to describe its image. We will usually write
$$Z:=\pi(V(R))$$
for this image set. This is a \sa\ subset of $W(R)$.

When $G(R)$ is \sa ally compact, the problem of determining $Z$ was
solved by Procesi and Schwarz \cite{PS}, complemented by Br\"ocker
\cite{Br}. Both worked in the case $R=\R$, but their main results
generalize to arbitrary real closed ground field:
\end{lab}

\begin{thm}[Procesi-Schwarz \cite{PS}, Br\"ocker \cite{Br}]
  \Label{procesischw}
Let $G$ be a linear group over $R$ acting on an affine $R$-variety
$V$, and assume that $G(R)$ is \sa ally compact. Let $\pi\colon V\to
V\qu G=W$ be the quotient morphism.
\begin{itemize}
\item[(a)]
The non-empty fibres of the map $\pi\colon V(R)\to W(R)$ are
precisely the $G(R)$-orbits in $V(R)$.
\item[(b)]
For every basic closed set $K$ in $V(R)$, the image $\pi(K)$ is a
basic closed subset of $W(R)$. In particular, $Z=\pi(V(R))$ is basic
closed in $W(R)$.
\item[(c)]
The map $\pi\colon V(R)\to W(R)$ is \sa ally proper.
\end{itemize}
\end{thm}

(Note that we are assuming that $G(R)$ is Zariski dense in $G$, as
always.)

\begin{lab}
In the case $R=\R$, and for the particular case $K=V(\R)$ in (b),
this theorem is due to Procesi and Schwarz \cite{PS}. Br\"ocker
\cite{Br} worked over $\R$ as well. He extended (b) to the case of
arbitrary basic closed $K$ (Prop.\ 5.1), and otherwise studied the
question of how many inequalities are needed to describe the basic
closed set $\pi(K)$.

Both \cite{PS} and \cite{Br} used transcendental arguments in their
proofs (in particular, integration). Therefore the proofs do not
directly generalize to other real closed ground fields. The proof of
(b) from \cite{Br} is easy to transfer to $R\ne\R$, once basic
properties of the Reynolds operator have been established; we'll give
it in \ref{imbc} below. Assertion (c) follows in a standard way from
(a) and~(b).%
\footnote
  {It is well-known that a \sa\ map is \sa ally proper if it maps
  closed \sa\ sets to closed \sa\ sets and if its fibres are \sa ally
  compact.}
As to assertion (a), it is possible to give an entirely algebraic
argument valid over any real closed $R$. However, this uses much
more of structure theory of linear algebraic groups and of invariant
theory than the rest of this paper, and therefore we omit it. Instead
we simply remark that (a) can be proved over $R\ne\R$ by deducing it
from the case $R=\R$ via the Tarski principle.
\end{lab}

\begin{lab}\Label{ineqorb}
A central point in the work of Procesi and Schwarz is the fact that
the inequalities for the image $Z$ of $\pi\colon V(R)\to W(R)$ can be
found constructively. Let us briefly recall how this is done. For
simplicity, assume that $V$ is a linear representation of $G$,
defined over $R$. Since $G(R)$ is \sa ally compact, there exists a
$G$-invariant positive definite inner product $\langle\hvar,\hvar
\rangle$ on $V$. Choose orthonormal linear coordinates $x_1,\dots,
x_n$ on $V$. Putting
$$dp=\sum_{k=1}^n\frac{\partial p}{\partial x_k}\>dx_k$$
for $p\in R[V]$, and transferring the inner product $\langle\hvar,
\hvar\rangle$ to the cotangent bundle of $V$
by $\langle dx_i,dx_j\rangle=\delta_{ij}$ ($i$, $j=1,\dots,n$), we
have
$$\langle dp,dq\rangle=\sum_{k=1}^n\frac{\partial p}
{\partial x_k}\cdot\frac{\partial q}{\partial x_k}$$
for $p$, $q\in R[V]$. Let $p_1,\dots,p_m$ be generators of $R[V]^G$,
the $R$-algebra of $G$-invariants. The inner products $\langle dp_i,
dp_j\rangle$ ($i$, $j=1,\dots,m$) are $G$-invariant,
and so the symmetric matrix
$$M=\bigl(\langle dp_i,dp_j\rangle\bigr)_{i,j=1,\dots,m}$$
has entries in $R[V]^G$. Procesi-Schwarz proved that
$$\pi(V(R))=\{z\in W(R)\colon M(z)\ge0\}.$$
(Here $M(z)\ge0$ means that the matrix $M(z)$ is positive
semidefinite.) Since $M(z)\ge0$ if and only if all the $2^m-1$
principal minors of $M(z)$ are $\ge0$, this shows that the \sa\ set
$Z=\pi(V(R))$ is basic closed, and can be described by $2^m-1$
non-strict inequalities.

The actual minimal number of inequalities needed for the description
of $Z$ may be much smaller. For results in this direction see
\cite{Br}.
\end{lab}

\begin{lab}\Label{ginvsos}
We claim that the principal minors of the matrix $M$ above are sums
of squares in $R[V]$. Indeed, $M=JJ^t$ where $J=Jac(p_1,\dots,p_m)=
\bigl(\frac{\partial p_i}{\partial x_k}\bigr)$ is the Jacobian matrix
of $(p_1,\dots,p_m)$. From the Binet-Cauchy theorem (see \cite{Ga},
for example)
it follows that the principal minors of any matrix of the
form $JJ^t$ are sums of squares in the ring generated by the
coefficients of $J$.

Thus, the principal minors of $M$ belong to the cone $S_0=(\Sigma
R[V]^2)^G$ of $G$-invariant sums of squares. Below, we will study
this cone more closely.
\end{lab}

\begin{lab}\Label{defreynolds}
The Reynolds operator is an essential tool for working with actions
of reductive groups. For later reference we collect its definition
and a few basic facts.

Let $G$ be a reductive group over $R$, and let $M$ be a $G$-module
which is finite-dimensional or, more generally, a union of
finite-dimensional submodules (such as the coordinate ring $R[V]$ of
an affine $G$-variety $V$). Since $G$ is reductive, the module $M^G$
of $G$-invariants in $M$ has a $G$-invariant complement $N$, and it
is immediate to see that $N$ is unique. The \emph{Reynolds operator}
$$\rey=\rey_M\colon M\to M^G$$
is defined to be the projection onto the $G$-invariants along the
direct sum decomposition $M=M^G\oplus N$.
\end{lab}

\begin{lab}\Label{rey2props}
For any affine $G$-variety $V$ we have, in particular, the Reynolds
operator
$$\rey=\rey_V\colon R[V]\to R[V]^G=R[V\qu G].$$
The map $\rey$ is $R$-linear and is uniquely characterized by the
following two properties:
\begin{itemize}
\item[(1)]
$\rey(f)=f$ for every $f\in R[V]^G$,
\item[(2)]
$\rey(f^g)=\rey(f)$ for every $f\in R[V]$ and $g\in G(R)$.
\end{itemize}
(This characterization uses that $G(R)$ is Zariski dense in $G$.) See
\cite{DK}, 2.2.2. An important property of $\rey$ is that it is
$R[V]^G$-linear, i.~e., that
$$\rey(a\cdot f)=a\cdot\rey(f)$$
holds for every $a\in R[V]^G$ and $f\in R[V]$.
\end{lab}

For the calculation of $\rey_V$ it suffices to know $\rey_G\colon
R[G]\to R$, the Reynolds operator of $G$ acting on itself by
translation, thanks to the following lemma:

\begin{lem}\Label{calcrey}
Let $G$ be reductive, and let $V$ be an affine $G$-variety. If
$f\in R[V]$ and $\eta_V(f)=\sum_{i=1}^ma_i\otimes f_i$, then
$$\rey_V(f)=\sum_{i=1}^m\rey_G(a_i)\cdot f_i.$$
\end{lem}

\begin{proof}
For $f\in R[V]$ with $\eta_V(f)=\sum_ia_i\otimes f_i$ define $r(f):=
\sum_i\rey_G(a_i)\cdot f_i$. Then $r\colon R[V]\to R[V]$ is an
$R$-linear map,
and we'll show that $r$ satisfies properties (1) and
(2) of \ref{rey2props}. If $f\in R[V]^G$ then $\eta_V(f)=1\otimes
f$, and so $r(f)=f$. It remains to show $r(f^g)=r(f)$ for $f\in R[V]$
and $g\in G(R)$.
By the commutative square
$$\xymatrix@C=+40pt{
R[V] \ar[r]^{\eta_V} \ar[d]_{\eta_V} & R[G]\otimes R[V]
  \ar[d]^{\id\otimes\eta_V} \\
R[G]\otimes R[V] \ar[r]^{\eta_G\otimes\id} &
  R[G]\otimes R[G]\otimes R[V]}$$
we have $\sum_ia_i\otimes\eta_V(f_i)=\sum_i\eta_G(a_i)\otimes f_i$
(in $R[G]\otimes R[G]\otimes R[V]$). Applying the homomorphism
$a\otimes b\otimes h\mapsto a(g)\cdot b\otimes h$ to both sides we
get
$$\eta_V\Bigl(\sum_ia_i(g)\,f_i\Bigr)=\sum_ia_i(g)\cdot\eta_V(f_i)=
\sum_ia_i^g\otimes f_i$$
(in $R[G]\otimes R[V]$). Since $f^g=\sum_ia_i(g)\,f_i$, we conclude
$$r(f^g)=\sum_i\rey_G(a_i^g)\,f_i=\sum_i\rey_G(a_i)\,f_i=r(f).$$
\end{proof}

If $R=\R$ and the group $G(\R)$ is compact, the Reynolds operator is
just averaging over the $G(\R)$-orbits (c.f.\ \cite{DK}, p.~45):

\begin{prop}\Label{reyhaar}
Let $G(\R)$ be compact, and let $V$ be an affine $G$-variety. Then
the Reynolds operator on $V$ is characterized by
$$(\rey_Vf)(x)=\int_{G(\R)}f(gx)\>d\lambda(g)$$
($f\in\R[V]$, $x\in V(\R)$). Here $\lambda$ is the normalized Haar
measure on $G(\R)$.
\end{prop}

\begin{rem}\Label{reyfinite}
In particular, if the group $G$ is finite, then
$$\rey_V(f)=\frac1{|G|}\sum_{g\in G(R)}f^g$$
for every $f\in R[V]$. Of course, this is true for any $R$, and is
not restricted to $R=\R$.
\end{rem}

\begin{lem}\Label{reyequivar}
For every equivariant morphism $\varphi\colon V\to V'$ of affine
$G$-varieties, the square
$$\xymatrix{
R[V'] \ar[r]^{\varphi^*} \ar[d]_{\rey_{V'}} & R[V] \ar[d]^{\rey_V} \\
R[V']^G \ar[r]^{\varphi^*} & R[V]^G}$$
commutes.
\end{lem}

\begin{proof}
There are unique $G$-submodules $M$ of $R[V]$ and $M'$ of $R[V']$
such that $R[V]=R[V]^G\oplus M$ and $R[V']=R[V']^G\oplus M'$. The map
$\varphi^*\colon R[V']\to R[V]$ is a $G$-module homomorphism and
hence respects these decompositions. The claim follows from this.
\end{proof}

\begin{lab}\Label{dfnsemiinv}
Let the reductive group $G$ act on the affine $R$-variety $V$, and
let $\omega$ be the isomorphism type of a finite-dimensional
irreducible $G$-module. A polynomial $f\in R[V]$ is said to be
\emph{semi-invariant} of type $\omega$ if the $G$-submodule of $R[V]$
generated by $f$ is isomorphic to a direct sum of copies of $\omega$.
\end{lab}

The following generalizes a theorem by Gatermann and Parrilo
(\cite{GP}, Theorems 5.3 and 6.2):

\begin{thm}\Label{gatparr}
Let $V$ be an affine $G$-variety, and assume that $G(R)$ is \sa ally
compact. Let $f\in R[V]$ be a $G$-invariant sum of squares. Then $f$
can be written
$$f=f_1^2+\cdots+f_m^2$$
in such a way that every $f_i\in R[V]$ is semi-invariant (of some
type~$\omega_i$).
\end{thm}

In short: An invariant sum of squares is a sum of squares of
semi-invariants. (Note that it is \emph{not} true conversely that
every sum of squares of semi-invariants is invariant.)

The main result of \cite{GP} corresponds to the case of \ref{gatparr}
where $G$ is finite and $V$ is a linear representation of $G$.

\begin{lab}\Label{gramethod}
The proof follows the ideas of \cite{GP}. However, we have translated
them into a more conceptual and less matrix-based setting. It uses
what has been called the ``Gram matrix method'' \cite{PW}, for the
characterization of sums of squares of polynomials. We start by
reviewing and rephrasing this approach.

At the beginning, $k$ can be any field of characteristic
not two. Let $V$ be a finite-dimensional $k$-vector space, and let
$$S(V)=\bigoplus_{n\ge0}S^n(V)$$
be the symmetric algebra of $V$, where $S^n(V)$ is the $n$-th
symmetric power of $V$. Let $V^\du$ be the dual vector space of $V$.
The symmetric algebra
$$S(V^\du)=\bigoplus_{n\ge0}S^n(V^\du)$$
of $V^\du$ is canonically identified with the ring of polynomials on
$V$, namely, $S^n(V^\du)$ is identified with the homogeneous
polynomials of degree $n$.

Fix a degree $d\ge0$. Any (symmetric) bilinear form
$$\gamma\colon S^d(V)\times S^d(V)\to k$$
determines a homogeneous polynomial $p_\gamma$ on $V$ of degree $2d$,
that is, an element of $S^{2d}(V^\du)$. Indeed, $\gamma$ is an
element of
$$(S^d(V)\otimes S^d(V))^\du=S^d(V)^\du\otimes S^d(V)^\du=S^d(V^\du)
\otimes S^d(V^\du),$$
and applying the product map $S^d(V^\du)\otimes S^d(V^\du)\to S^{2d}
(V^\du)$ to this element yields $p_\gamma$. As a map from $V$ to $k$,
$p_\gamma$ is given by
$$p_\gamma(v)=\gamma(v^d,v^d)\quad(v\in V),$$
where $v^d$ denotes the $d$-th power in the symmetric algebra
$S(V)$. In analogy to the terminology in \cite{PW}, we'll say that
$\gamma$ is a \emph{Gram form} for the homogeneous polynomial $f$ (of
degree $2d$) if $p_\gamma=f$.

Now assume that $k=R$ is a real closed field. Then a homogeneous
polynomial $f\in S^{2d}(V^\du)$ (of degree $2d$) is a sum of squares
of (homogeneous) polynomials (of degree $d$) if and only if $f$ has a
Gram form which is positive semidefinite (psd). In fact, $f$ is a sum
of $r$ squares if and only if $f$ has a psd Gram form of rank $\le
r$. Note that the set of Gram forms of $f$ is an affine-linear
subspace of $S^2(S^d(V^\du))$, the space of symmetric bilinear forms
on $S^d(V)$. The set of psd Gram forms of $f$ is therefore a closed
convex \sa\ subset of $S^2(S^d(V^\du))$.
\end{lab}

\begin{lab}
We now give the proof of Theorem \ref{gatparr}.
Since $V$ has an equivariant closed embedding into a linear
representation $W$ of $G$, and since $R[W]^G\to R[V]^G$ is
surjective, it is clear that we can assume that $V$ is a linear
representation of $G$, i.~e.\ a finite-dimensional $R$-vector space
with a linear $G$-action. We will give the proof for forms, i.~e.,
homogeneous polynomials. This is not a restriction of generality.

Thus assume $f$ is a homogeneous polynomial on $V$ of degree $2d$
which is $G$-invariant and which is a sum of squares of polynomials.
Thus $f$ has a psd Gram form $\gamma\in S^2(S^d(V^\du))$, which is a
psd symmetric bilinear form
$$\gamma\colon S^d(V)\times S^d(V)\to R$$
with $p_\gamma=f$ (see \ref{gramethod}). Consider the action of $G$
on $S^2(S^d(V^\du))$ which is induced by its action on $V$. For every
$g\in G(R)$, $\gamma^g$ is again a psd Gram form of $f$. Since the
set of psd Gram forms of $f$ is convex, it contains a $G$-invariant
element by Proposition \ref{reyconvex} below.%
\footnote
  {This forward reference is for ease of exposition only and does not
  create a logical circle.}
For what follows, we can therefore assume that $\gamma$ is
$G$-invariant.

Consider the decomposition of the $G$-module $S^d(V)$ into isotypical
components:
\begin{equation}\Label{isotyp}
S^d(V)=M_1\oplus\cdots\oplus M_r.
\end{equation}
Thus the $M_i$ are $G$-invariant, and $\Hom_G(M_i,M_j)=0$ for $i\ne
j$. The group $G(R)$ being \sa ally compact, every irreducible
representation of $G$ is self-dual, i.~e., isomorphic to its dual.
(Choose a $G$-invariant positive inner product to see this.) Hence
$M_i^\du\cong M_i$ as $G$-modules, for each $i$.

The decomposition \eqref{isotyp} is orthogonal with respect to
$\gamma$, that is, $\gamma(x,y)=0$ for all $x\in M_i$ and $y\in M_j$
whenever $i\ne j$. Indeed, by the $G$-invariance of $\gamma$ we have
$\gamma(x,y)=\gamma(gx,gy)$ for all $x$, $y\in S^d(V)$ and $g\in
G(R)$. In particular, the  linear map $M_j\to M_i^\du$ induced by the
restriction of $\gamma$ to $M_i\oplus M_j$ is $G$-equivariant. Since
$M_i^\du\cong M_i$ as $G$-modules, as mentioned before, it follows
that $M_j\to M_i^\du$ must be the zero map, which means that $M_i$
and $M_j$ are orthogonal.

Diagonalizing the restriction of $\gamma$ to each $M_i$ separately,
we see from this that $f$ can be written as a sum of squares of
semi-invariant polynomials.
\qed
\end{lab}


\section{Quadratic modules and \sa\ sets in the orbit variety}

\begin{lab}
When studying quadratic modules $M$ in the coordinate ring of an
affine $R$-variety $V$, it is usually necessary to work with the
(pro-basic and closed) subsets $\scrX(M)$ in the real spectrum of
$R[V]$, rather than with their traces $\scrS(M)$ in $V(R)$ (see
\ref{notat} for the notation), unless one knows that the quadratic
module $M$ is finitely generated. Since we cannot always assume this,
we are using the real spectrum.

To have a notation available which is less awkward, let us introduce
the following shorthands. Let $V$ be an affine $R$-variety with
coordinate ring $R[V]$, and let $M$ be a quadratic module in $R[V]$.
We write
$$\scrX_V(M):=\scrX_{R[V]}(M)$$
for the closed subset of $\wt{V(R)}=\Sper R[V]$ which is associated
with $M$. Given a subset $X$ of $\wt{V(R)}$, we write
$$\scrP_V(X):=\scrP_{R[V]}(X)$$
for the cone of elements in $R[V]$ that are non-negative on $X$.
\end{lab}

\begin{lab}
In the following, let always $G$ be a reductive group over $R$, let
$V$ be an affine $G$-variety and $\pi\colon V\to V\qu G=W$ the
quotient morphism. By $\pi$ we also denote the induced map $\pi\colon
V(R)\to W(R)$ on real points. We will always write $Z:=\pi(V(R))$ for
the image set of $\pi\colon V(R)\to W(R)$. This is a \sa\ subset of
$W(R)$. By a theorem of Luna \cite{Lu}, $Z$ is closed in $W(R)$. (We
will not use this fact.) Accordingly,
$$\wt\pi\colon\Sper R[V]=\wt{V(R)}\to\wt{W(R)}=\Sper R[W]$$
denotes the associated map of real spectra. The image of $\wt\pi$ is
$\wt Z$.
\end{lab}

The following observations are pure formalities and have nothing to
do with the specific situation:

\begin{lem}\Label{lem1po}
Let $M$ be an arbitrary quadratic module in $R[V]$. Put $N:=M\cap
R[W]$, and write $X:=\scrX_V(M)$ and $Y:=\scrX_W(N)$.
\begin{itemize}
\item[(a)]
$\wt\pi(X)\subset Y$.
\item[(b)]
If $M$ is saturated in $R[V]$, then $N$ is saturated in $R[W]$. In
fact, $N=\scrP_W(\tilde\pi(X))$, and hence $\wt\pi(X)=Y$.
\item[(c)]
Equality $X=\wt\pi^{-1}(Y)$ holds if and only if $M\subset\Sat_V(N)$.
In particular, then, $\wt\pi(X)=\wt Z\cap Y$.
\end{itemize}
\end{lem}

If $M$ is saturated in $R[V]$ then (b) implies that $Y$ is the
pro-basic closed hull of $\wt\pi(X)$. One instance when the
equivalent conditions of (c) are satisfied is when $M$ is generated
by elements of $R[W]$, as a quadratic module in $R[V]$.

\begin{proof}
Denote the inclusion $R[W]\subset R[V]$ by $i$. (a) is obvious.
To prove (b), note that (a) says $N\subset\scrP_W(\wt\pi(X))$.
Conversely let $b\in\scrP_W(\wt\pi(X))$. Then $i(b)\ge0$ on $X$, so
$i(b)\in\scrP_V(X)=M$, which means $b\in N$.
To prove (c), note that $\scrX_V(N)=\wt\pi^{-1}(Y)$. From $N\subset
M$ it follows that $\Sat_V(N)\subset\Sat_V(M)$. Hence $M\subset\Sat_V
(N)$ $\iff$ $\Sat_V(M)=\Sat_V(N)$ $\iff$ $X=\wt\pi^{-1}(Y)$.
\end{proof}

For more interesting results, we have to assume that $G(R)$ is
\sa ally compact.

\begin{thm}\Label{invconvex}
Assume that $G(R)$ is \sa ally compact. Then for every $G$-invariant
convex subset $C$ of $R[V]$, the Reynolds operator $\rey_V$ satisfies
$\rey_V(C)=C\cap R[V]^G$. In particular, if $C$ is non-empty then $C$
contains a $G$-invariant element.
\end{thm}

Since $\rey_V$ is the identity on $R[V]^G$, the inclusion $C\cap R[V]
^G\subset\rey_V(C)$ is trivial, and we only have to show $\rey_V(C)
\subset C$. The proof is obvious when $G$ is finite, since then $\rey
_V(f)$ is a convex combination of the finitely many $G$-translates
of $f$ (\ref{reyfinite}). In the general case, invariant integration
may replace this argument when $R=\R$, but not for other $R$.

Instead we use the following uniform argument. Let $f\in C$. There is
a finite-dimensional $G$-invariant subspace $U$ of $R[V]$ containing
$f$. So the theorem follows from the following

\begin{prop}\Label{reyconvex}
Assume that $G(R)$ is \sa ally compact, and let $U$ be a
finite-dimensional $G$-module. Let $\rey_U\colon U\to U^G$ be the
Reynolds operator of the $G$-module $U$. Then $\rey_U(u)$ lies in the
convex hull of the orbit $G(R)\,u$, for every $u\in U$.
\end{prop}

In particular, if $C$ is a $G$-invariant convex subset of $U$, then
$\rey_U(C)=C\cap U^G$.

\begin{proof}
There exists a $G$-invariant positive definite inner product
on $U$. By Cara\-th\'eo\-dory's lemma, the convex hull $C$ of the
orbit $G(R)\,u$ is \sa ally compact.
Hence there exists a unique point $v$ in $C$ of minimal distance to
the origin. Clearly, $v$ must be $G$-invariant.
Being a convex combination of finitely many translates $gu$, $g\in
G(R)$, it is clear that in fact $v=\rey(u)$.
\end{proof}

\begin{cor}\Label{reyinvmod}
Let $G(R)$ be \sa ally compact, let $V$ be an affine $G$-variety,
and let $W=V\qu G$. Let $\rey=\rey_V\colon R[V]\to R[W]$.
\begin{itemize}
\item[(a)]
$\rey\bigl(\Sigma R[V]^2\bigr)=R[W]\cap\Sigma R[V]^2=\bigl(\Sigma
R[V]^2\bigr)^G$. This is a preordering in $R[W]$, and we denote it by
$S_0$.
\item[(b)]
If $M$ is any quadratic module in $R[V]$, then $\rey(M)$ is an
$S_0$-module in $R[W]$.
\item[(c)]
If $M$ is a $G$-invariant quadratic module in $R[V]$, then $\rey(M)=
M\cap R[W]$.
\end{itemize}
\end{cor}

\begin{proof}
For any $G$-invariant quadratic module $M$ in $R[V]$ we have $\rey(M)
=M\cap R[W]$ by Theorem \ref{invconvex}. This proves (a) and (c).
To prove (b), let $M$ be an arbitrary quadratic module in $R[V]$.
Clearly $\rey(M)$ is additively closed. For every $f\in S_0$ we have
$fM\subset M$ since $f\in\Sigma R[V]^2$. Hence $f\cdot\rey(M)=\rey
(fM)\subset\rey(M)$.
\end{proof}

\begin{cor}
Let $G(R)$ be \sa ally compact, and let $V$ be an affine $G$-variety.
The elements of $S_0=\bigl(\Sigma R[V]^2\bigr)^G$ are precisely the
finite sums of elements of the form $\rey(a^2)$ with $a\in R[V]$
semi-invariant (c.~f.\ \ref{dfnsemiinv}).
\end{cor}

\begin{proof}
Let $f\in S_0$, so $f$ is a $G$-invariant sum of squares in $R[V]$.
By Theorem \ref{gatparr} we can write $f=f_1^2+\cdots+f_r^2$ with
each $f_i$ semi-invariant. Therefore
$$f=\rey(f)=\rey(f_1^2)+\cdots+\rey(f_r^2).$$
\end{proof}

\begin{Rems}\Label{rhomex}
In Section~5, we will give examples showing the following:
\begin{enumerate}
\item
The preordering $S_0$ in $R[W]$ need not be finitely generated
(Example \ref{s0notfg}).
\item
If $T$ is a finitely generated preordering in $R[V]$, then the
$S_0$-module $\rey(T)$ need not be finitely generated, not even if
$T$ is $G$-invariant (Example \ref{reymnotfgs0mod}).
\item
If $T$ is a preordering in $R[V]$, then the quadratic module $\rey
(T)$ in $R[W]$ need not be a preordering (Example \ref{reynotpo}).
\end{enumerate}
All these examples have the simplest possible group acting, namely
the group $G$ of order two.
\end{Rems}

\begin{cor}\Label{basic}
Let $G(R)$ be \sa ally compact, and let $K$ be a $G$-invariant subset
of $V(R)$. If $f\in R[V]$ satisfies $f(x)\ge0$ for every $x\in K$,
then $(\rey f)(\pi x)\ge0$ for every $x\in K$. In other words,
$$\scrP_V(K)^G=\rey\bigl(\scrP_V(K)\bigr)=\scrP_W(\pi K).$$
\end{cor}

\begin{proof}
$\scrP_V(K)$ is a $G$-invariant convex subset of $R[V]$, and $\scrP_V
(K)\cap R[W]=\scrP_W(\pi K)$. Therefore, the corollary is a special
case of Theorem \ref{invconvex}. Alternatively, it follows from
\ref{lem1po}(b) applied to $M=\scrP_V(K)$, using \ref{reyinvmod}(c).
\end{proof}

We now recall Br\"ocker's result that $\pi\colon V(R)\to W(R)$ maps
$G$-invariant basic closed sets to basic closed sets. The proof in
\cite{Br} uses integration, and so it works only for $R=\R$. Using
\ref{basic}, however, there is no difficulty in giving a proof which
works in general:

\begin{prop}\Label{imbc}
Assume that $G(R)$ is \sa ally compact. Let $f_1,\dots,f_r\in R[V]$,
and let $K:=\scrS_V(f_1,\dots,f_r)$ and $T:=PO_V(f_1,\dots,f_r)$.
If $K$ is $G$-invariant then the \sa\ set $\pi(K)$ in $W(R)$ is basic
closed. In fact,
$$\wt{\pi(K)}=\scrX_W(\rey(T))=\bigcap_{a\in R[V]}\bigcap_
{i\in\{0,1\}^r}\scrX_W\bigl(\rey(a^2f_1^{i_1}\cdots f_r^{i_r})
\bigr),$$
and equality holds for a finite sub-intersection of the right hand
intersection.
\end{prop}

\begin{proof}
(C.f.\ \cite{Br} Prop.~1.2, Prop.~5.1)
$T\subset\scrP_V(K)$ implies $\rey(T)\subset\rey(\scrP_V(K))=\scrP_W
(\pi K)$ (\ref{basic}), and thus $\pi(K)\subset\scrX_W\bigl(\rey(T)
\bigr)$.
Conversely let $\beta\in\wt{W(R)}$, $\beta\notin\wt{\pi(K)}$. Thus
$\wt\pi^{-1}(\beta)\cap\scrX_V(T)=\emptyset$. By an application of
the general Stellensatz (\cite{BCR} p.~91, \cite{KS} p.~143), this
means that the preordering of $R[V]$ generated by $P_\beta:=\{b\in
R[W]\colon b(\beta)\ge0\}$ and by $T$ contains an element $-b$ with
$b\in R[W]$, $b(\beta)>0$. So there is an identity
$$-b=\sum_{j=1}^mb_jt_j$$
in $R[V]$, where $b\in R[W]$ satisfies $b(\beta)>0$, and where $b_j
\in P_\beta$ and $t_j\in T$. Applying $\rey$ gives $-b=\sum_jb_j\rey
(t_j)$, which shows that $\rey(t_j)(\beta)<0$ for some $j$. Since
$t_j$ is a sum of finitely many products $a^2f_1^{i_1}\cdots f_r^
{i_r}$ with $i_\nu\in\{0,1\}$ and $a\in R[V]$, we see that $\beta$ is
not contained in the right hand intersection.

A simple compactness argument implies that $\wt{\pi(K)}$ is equal to
a finite subintersection of the double intersection. Namely, the sets
in the double intersection are clearly constructible, while
$\wt{\pi(K)}$ is constructible because $\pi(K)$ is \sa\ by the
Tarski-Seidenberg Theorem. The claim now follows from the fact that
the constructible sets are closed and compact in the constructible
topology.
\end{proof}

\begin{cor}\Label{prop1}
Assume that $G(R)$ is \sa ally compact, and let $S_0=\rey(\Sigma
R[V]^2)$ as above. Then $\scrX_W(S_0)=\wt Z$.
\end{cor}

\begin{proof}
This is the particular case $T=\Sigma R[V]^2$ of Proposition
\ref{imbc}, see \ref{reyinvmod}. A stronger (explicit) result can be
obtained from \ref{ineqorb} and \ref{ginvsos}.
\end{proof}

The next result generalizes this corollary (\ref{prop1} corresponds
to the case $N=\Sigma R[W]^2$ and $M=\Sigma R[V]^2$):

\begin{prop}\Label{lem2po}
Let $G(R)$ be \sa ally compact. Let $N$ be a quadratic module in
$R[W]$, and let $M$ be the quadratic module which is generated by
$N$ in $R[V]$.
\begin{itemize}
\item[(a)]
$\rey(M)=M\cap R[W]$, and this is the $S_0$-module generated by $N$
in $R[W]$.
\item[(b)]
$\wt\pi(\scrX_V(M))=\scrX_W(N)\cap\wt Z=\scrX_W(M\cap R[W])$.
\item[(c)]
In particular, if $N$ is finitely generated, then $\scrX_W(M\cap
R[W])$ is constructible in $\wt{W(R)}$.
\end{itemize}
\end{prop}

\begin{proof}
(a)
Clearly, $M$ is $G$-invariant. Hence $\rey(M)=M\cap R[W]$, and this
is an $S_0$-module (\ref{reyinvmod}). Every
element $f\in M$ can be written $f=\sum_ia_i^2g_i$ with $g_i\in N$
and $a_i\in R[V]$. Hence
$$\rey(f)=\sum_i\rey(a_i^2)\,g_i,$$
which shows that $\rey(M)$ is contained in the $S_0$-module generated
by $N$.

For the proof of (b) note that $\scrX_V(M)=\wt\pi^{-1}(\scrX_W(N))$,
from which we get the first equality in (b). By (a), $\scrX_W(M\cap
R[W])=\scrX_W(N)\cap\scrX_W(S_0)$, and combined with \ref{prop1} this
gives the second equality. (c) is obvious from (b).
\end{proof}

\begin{cor}
A quadratic module $N$ in $R[W]$ is of the form $N=M\cap R[W]$ for
some quadratic module $M$ in $R[V]$, if and only if $N$ is an
$S_0$-module.
\qed
\end{cor}

\begin{proof}
If $N=M'\cap R[W]$ with some quadratic module $M'$ in $R[V]$, then
also $N=M\cap R[W]$ with $M$ the quadratic module generated by $N$ in
$R[V]$, and so $N$ is an $S_0$-module by \ref{lem2po}(a). Conversely,
if $N$ is an $S_0$-module, let $M$ be the quadratic module generated
by $N$ in $R[V]$; then $N=M\cap R[W]$, again by \ref{lem2po}(a).
\end{proof}

\begin{rem}\Label{exinvdescrp}
Proposition \ref{imbc} implies in particular that every $G$-invariant
basic closed set $K\subset V(R)$ has a description $K=\scrS_V(h_1,
\dots,h_m)$ by $G$-invariant functions $h_1,\dots,h_m\in R[V]$. More
precisely, if $K=\scrS_V(f_1,\dots,f_r)$ with arbitrary $f_i\in
R[V]$, then the $h_j$ can be chosen of the form $\rey(a^2f_1^{i_1}
\cdots f_r^{i_r})$ with $a\in R[V]$ and $i_\nu\in\{0,1\}$.

On the other hand, it is not true in general that $K=\scrS_V
(\rey f_1,\dots,\rey f_r)$. (Only ``$\subset$'' holds in general, by
\ref{basic}). For example, consider $G=\mu_2$ acting on the line $V=
\A^1$ through multiplication by $-1$, and take $K=[-1,1]=\scrS_V
(f_1,f_2)$, where $f_1=1+x$, $f_2=1-x$. Then $\rey f_i=1$ ($i=1,2$),
so $\scrS_V(\rey f_1,\rey f_2)=R\ne K$.

Therefore the question arises how $G$-invariant functions $h_1,\dots,
h_r$ with $K=\scrS_V(h_1,\dots,h_r)$ can be found concretely. Here we
give an answer in the case where $G$ is finite:
\end{rem}

\begin{prop}\Label{thm1}
Assume $|G|=n$ is finite, and let $f_1,\dots,f_r\in R[V]$ be such
that the set $K=\scrS_V(f_1,\dots,f_r)$ is $G$-invariant. For $i=1,
\dots,r$ and $j=1,\dots,n$ let $s_{ij}$ be the $j$-th elementary
symmetric function in the $n$ elements $f_i^g$ ($g\in G$). Then
$s_{ij}\in R[V]^G$, and $K$ is the subset of $V(R)$ where the $nr$
functions $s_{ij}$ are non-negative.
\end{prop}

\begin{proof}
Given real numbers $x_1,\dots,x_n$, let $s_i=s_i(x_1,\dots,x_n)$ be
the $i$-th elementary symmetric function of the $x_j$. Then
$$x_1,\dots,x_n\ge0\quad\iff\quad s_1,\dots,s_n\ge0.$$
Indeed, the right hand condition implies $\prod_j(x-x_j)=x^n-s_1
x^{n-1}+\cdots+(-1)^ns_n\ne0$ for any $x<0$, since all summands have
the same sign. Hence for every index $i=1,\dots,r$,
$$\bigcap_{g\in G}\{f_i^g\ge0\}=\bigcap_{j=1}^n\{s_{ij} \ge 0\},$$
and so $K=\bigcap_i\bigcap_{g\in G}\{f_i^g\ge0\}
=\bigcap_i\bigcap_j\{s_{ij}\ge0\}$.
\end{proof}

\begin{rem}
For any basic closed and $G$-invariant subset $K$ of $V(R)$,
Proposition \ref{thm1}, combined with explicit inequalities for
$\pi(V(R))$ (\ref{ineqorb}), gives a constructive way for obtaining
inequalities describing $\pi(K)$ in the orbit variety.
\end{rem}

We do not know whether a similar constructive procedure exists for
the case $G(R)$ \sa ally compact, but infinite.


\section{Reynolds operator and sums of squares in the non-compact
  case}

Let $G$ be a reductive group over $R$ and $V$ an affine $G$-variety,
with Reynolds operator $\rey\colon R[V]\to R[V]^G$. If $G(R)$ is
\sa ally compact, we have seen that $\rey(f^2)$ is a sum of squares
in $R[V]$, for every $f\in R[V]$, and in turn, that
$$\rey\bigl(\Sigma R[V]^2\bigr)=\bigl(\Sigma R[V]^2\bigr)^G$$
(see \ref{reyinvmod}). In this section we will prove that this key
property fails (for suitable $V$) whenever $G(R)$ is not \sa ally
compact. First, we need two lemmas:

\begin{lem}\Label{reysos}
Let $G$ be a reductive group over $R$, let $\rey_G\colon R[G]\to R$
be the Reynolds operator. The following conditions are equivalent:
\begin{itemize}
\item[(i)]
$\rey_G(f^2)\ge0$ for every $f\in R[G]$;
\item[(ii)]
for every affine $G$-variety $V$ one has $\rey_V(\Sigma R[V]^2)
\subset\Sigma R[V]^2$.
\end{itemize}
For brevity, we will say that $G$ has property $(\star)$ if (i) and
(ii) hold.
\end{lem}

\begin{proof}
Of course, (i) is a particular case of (ii). Assume that (i) holds,
and let $f\in R[V]$, with (say) $\eta_V(f)=\sum_{i=1}^ma_i\otimes
f_i$ in $R[G]\otimes R[V]$. Hence
$$\eta_V(f^2)=\sum_{i,j=1}^m(a_ia_j)\otimes(f_if_j),$$
and therefore
$$\rey_V(f^2)=\sum_{i,j=1}^m\rey_G(a_ia_j)\cdot f_if_j$$
(Lemma \ref{calcrey}). The symmetric matrix $S:=(\rey_G(a_ia_j))_
{i,j=1,\dots,m}$ over $R$ is positive semidefinite, since for $c_1,
\dots,c_m\in R$ we have
$$\sum_{i,j=1}^mc_ic_j\,\rey_G(a_ia_j)=\rey_G\Bigl(\sum_{i,j=1}^m
c_ic_j\,a_ia_j\Bigr)=\rey_G\left(\Bigl(\sum_{i=1}^mc_i\,a_i\Bigr)^2
\right)\>\ge0$$
by hypothesis (i). After diagonalizing the matrix $S$, one sees
therefore that $\rey_V(f^2)$ is a sum of squares in $R[V]$.
\end{proof}

\begin{rem}
In \ref{reyinvmod} it was proved that property $(\star)$ holds
whenever $G(R)$ is \sa ally compact.
For $R=\R$, the usual real
numbers, there is an even easier transcendental proof, which uses
characterization (i) from \ref{reysos}. Indeed,
$$\rey_G(f^2)=\int_{G(\R)}f(g)^2\>dg\>\ge\>0$$
(see \ref{reyhaar}) is immediate for $f\in\R[G]$.
\end{rem}

\begin{example}\Label{reygm}
The multiplicative group $G=\G_m$ does not have property $(\star)$.
Indeed, let $R[\G_m]=R[x,x^{-1}]$ be its coordinate ring. The
Reynolds operator $\rey\colon R[\G_m]\to R$ is given by
$$\rey\colon\;\sum_{i\in\Z}a_i\,x^i\>\mapsto\>a_0$$
($a_i\in R$). In particular, $\rey\bigl((x-x^{-1})^2\bigr)=-2$,
and so property $(\star)$ fails.
\end{example}

\begin{lem}\Label{reystrpos}
Let $G$ be reductive, let $V$ be a homogeneous affine $G$-variety,
and consider the Reynolds operator $\rey=\rey_V\colon R[V]\to R$. If
$\rey(f^2)\ge0$ for every $f\in R[V]$, then actually $\rey(f^2)>0$
for every $f\ne0$.
\end{lem}

Recall that a $G$-variety $V\ne\emptyset$ is called
\emph{homogeneous} if  the group $G\bigl(R(\sqrt{-1})\bigr)$ acts
transitively on the set $V\bigl(R(\sqrt{-1})\bigr)$.

\begin{proof}
We first prove that the set
$$I:=\bigl\{f\in R[V]\colon\rey(f^2)=0\bigr\}$$
is an ideal in $R[V]$. Indeed, for $f\in I$ and $b\in R[V]$ we have
$$0\le\rey((b+tf)^2)=\rey(b^2)+2t\cdot\rey(bf)$$
for every $t\in R$, which implies $\rey(bf)=0$.
Hence $aI\subset I$ for every $a\in R[V]$ (take $b:=a^2f$). Moreover,
if $f$, $g\in I$ then $\rey((f\pm g)^2)=\pm2\,\rey(fg)\ge0$, hence
$\rey((f\pm g)^2)=0$, and so we have shown that $I$ is an ideal.

If $f\in I$ and $g\in G(R)$ then also $f^g\in I$, since $\rey
\bigl((f^g)^2\bigr)=\rey\bigl((f^2)^g\bigr)=\rey(f^2)=0$ by the
$G$-invariance of $\rey$. Therefore the closed subvariety of $V$
defined by $\sqrt I$ is invariant under translation by $G(R)$, and
hence must be equal to $V$ since $1\notin I$. Thus $I=(0)$.
\end{proof}

Now we can show:

\begin{thm}\Label{noncptnonstar}
Let $G$ be a reductive group over $R$. If $G$ has property $(\star)$
then $G(R)$ is \sa ally compact.
\end{thm}

The converse has already been proved in \ref{reyinvmod}.

\begin{proof}
We assume that $G(R)$ is non-compact and shall arrive at a
contradiction. Since $G(R)$ is not compact, $G$ contains a split
torus, i.e.\ a closed $R$-subgroup $H$ isomorphic to $\G_m$. Consider
the natural left action $(h,g)\mapsto hg$ of $H$ on $G$ by
translation,
and let $H\backslash G$ be the quotient variety. It is known
(\cite{KH} Thm.\ 5.1)
that $H\backslash G$ is an affine variety, with $R[H\backslash G]=
R[G]^H$, the ring of $H$-invariants in $R[G]$ with respect to this
action. Now $H\backslash G$ is a homogeneous $G$-variety for the
(left) action
$$G\times(H\backslash G)\to H\backslash G,\quad(g,Hx)\mapsto
Hxg^{-1},$$
and the Reynolds operator factors as
$$\xymatrix{
R[G] \ar[rr]^{\rey_H} \ar[dr]_{\rey_G} &&
  R[H\backslash G] \ar[dl]^{\rey_{H\backslash G}} \\
& R}$$
We will find an element $b\in R[G]$ such that $\rey_H(b^2)=-c^2$ for
some $c\in R[H\backslash G]$, $c\ne0$. This will be a contradiction.
For, on the one hand, $\rey_G(b^2)\ge0$ by hypothesis $(\star)$. On
the other, $\rey_G(b^2)=-\rey_{H\backslash G}(c^2)$ must be strictly
negative by Lemma \ref{reystrpos}.

We denote the dual action of $H$ on $G$ by
$$\eta'\colon R[G]\to R[H]\otimes R[G].$$
Let $X(H)$ be the character group of $H$ (an infinite cyclic group),
and let $u\in X(H)$ be a non-trivial character.
There exists an
element $x\in R[G]$, $x\ne0$, with
$$\eta'(x)=u\otimes x.$$
Indeed, if we decompose the $H$-module $R[G]$ into isotypical
components,
$$R[G]=\bigoplus_{\chi\in X(H)}R[G]_{(\chi)},$$
then $R[G]_{(\chi)}\ne0$ for every $\chi$,
and it suffices to take any $0\ne x\in R[G]_{(u)}$.
Similarly, choose $0\ne y\in R[G]$ with $\eta'(y)=u^{-1}\otimes y$.
Let $b:=x^2-y^2\in R[G]$. We have
$$\eta'(b^2)=u^4\otimes x^4-2\cdot1\otimes(x^2y^2)+u^{-4}\otimes
y^4.$$
Since the Reynolds operator $R[H]\to R$ of $H$ sends $u^n$ to zero
for all $n\ne0$ (Example \ref{reygm}), we get for $\rey_H\colon R[G]
\to R[H\backslash G]$, according to Lemma \ref{calcrey},
$$\rey_H(b^2)=-2\,(xy)^2=-c^2,$$
where $c:=\sqrt2\,xy$ lies in $R[H\backslash G]$, $c\ne0$. By the
argument given before, this completes the proof of Theorem
\ref{noncptnonstar}.
\end{proof}


\section{Three examples}

To motivate our first example, we recall a result of Procesi and
Schwarz for rational functions (\cite{PS} Sect.~7, slightly
generalized here):

\begin{prop}\Label{genbasic}
Let the reductive $R$-group $G$ act on the smooth irreducible affine
$R$-variety $V$. Let $\pi\colon V\to V\qu G=:W$ be the quotient
morphism, and let $Z:=\pi(V(R))\subset W(R)$. Equivalent conditions:
\begin{itemize}
\item[(i)]
The set $Z$ is generically basic;
\item[(ii)]
the preordering $T:=R(W)\cap\Sigma R(V)^2$ of the field $R(W)$ is
finitely generated.
\end{itemize}
In fact, if $Z$ is generically equal to $\scrS_W(p_1,\dots,p_m)$ with
$p_i\in R[W]$, then $T$ is generated by $p_1,\dots,p_m$ (as a
preordering of $R(W)$). Moreover, conditions (i) and (ii) are
satisfied when $G(R)$ is \sa ally compact.
\end{prop}


Recall that $Z$ generically basic means that there exists a basic
closed set $Z'$ in $W(R)$ such that the set-theoretic difference of
$Z$ and $Z'$ is not Zariski dense in $W$.

Note that $R(W)$ may be smaller than $R(V)^G$, the field of
$G$-invariant rational functions on $G$. Both coincide for all linear
representation spaces $V$ of $G$ defined over $R$, if and only if
every character $G\to\G_m$ defined over $R$ has finite image. In
particular, $R(W)=R(V)^G$ is always true for linear representations
$V$ if $G(R)$ is \sa ally compact. (These remarks are already made in
\cite{PS} 7.6.)

See \cite{PS} 7.8 for a class of representations $V$ where the
equivalent conditions of \ref{genbasic} fail.

\begin{lab}
Now consider the case where $G(R)$ is \sa ally compact. As we have
just recalled, the preordering $\bigl(\Sigma R(V)^2\bigr)^G$ of
$G$-invariant sums of squares of rational functions on $V$ is
finitely generated in the field $R(V)^G=R(W)$. One is therefore
wondering whether a similar result holds for regular functions. Thus,
is the preordering
$$S_0=\bigl(\Sigma R[V]^2\bigr)^G=\rey\bigl(\Sigma R[V]^2\bigr)$$
in $R[V]^G=R[W]$ finitely generated? It turns out that this usually
fails, as the following example shows.
\end{lab}

\begin{example}\Label{s0notfg}
Consider the group $G$ of order two acting on $V=\A^n$ by $x\mapsto
-x$. The ring of invariants $R[V]^G$ consists of all polynomials in
$R[V]=R[x_1,\dots,x_n]$ which contain only monomials of even degree.
Thus, $R[V]^G$ is generated as an $R$-algebra by $R[V]_2$, the space
of quadratic forms in $(x_1,\dots,x_n)$. Identifying $R[V]_2$ with
the space of symmetric $n\times n$-matrices over $R$, the cone $C:=
S_0\cap R[V]_2$ consists of all psd symmetric matrices. If $S_0$ were
finitely generated as a preordering in $R[V]^G$, then $C$ would be a
polyhedral cone in $R[V]_2$, i.e.\ we would have $C=R_\plus S_1+
\cdots+R_\plus S_m$ with finitely many psd matrices $S_\nu$. But this
is clearly not the case for $n\ge2$.

On the other hand, if the same $G$ acts on $\A^n$ instead by
$$(x_1,\dots,x_n)\mapsto(x_1,\dots,x_{n-1},-x_n),$$
then $R[V]^G$ is generated by $u_i=x_i$ ($i=1,\dots,n-1$) and $v=
x_n^2$, and $S_0$ is the preordering generated by $v$.
\end{example}

\begin{lab}
Let $M\subset R[V]$ be a quadratic module, generated by $f_1,\dots,
f_r$, say. From the example in Remark \ref{exinvdescrp} we know that
$\rey(M)$ can be larger than the $S_0$-module generated by $\rey
(f_1),\dots,\rey(f_r)$. However, the question remains if $\rey(M)$ is
at least finitely generated as an $S_0$-module. Our second example
shows that the answer is usually negative, even if $M$ is a
preordering and is $G$-invariant:
\end{lab}

\begin{example}\Label{reymnotfgs0mod}
Let $G=\mu_2$ act on $V=\A^2$ by interchanging the $x$ and $y$
coordinates, and let $T$ be the preordering in $R[V]=R[x,y]$ which
is generated by $x$ and $y$. Then $T$ is $G$-invariant, but the
$S_0$-module $\rey(T)$ fails to be finitely generated.

Assume to the contrary that $\rey(T)$ is finitely generated as an
$S_0$-module. Then $\rey(T)$ is generated as $S_0$-module by $1$,
$x+y$, $xy$ and by finitely many polynomials of the form $h_i=2\rey
(xg_i^2)$ with non-constant $g_i\in R[x,y]$ ($i=1,\dots,m$).
Consider the family of polynomials
$$f_r=2\rey\bigl((1-ry)^2x\bigr)=(x+y)-4rxy+r^2xy(x+y),$$
which lie in $\rey(T)$ for every value of the parameter $r\in R$. By
assumption, for every $r\in R$ there exist $s_0$, $s_1$, $s_2$, $t_1,
\dots,t_m\in S_0$ with
\begin{equation}\Label{ex2}
f_r=s_0+s_1\,(x+y)+s_2\,xy+\sum_{i=1}^mt_ih_i.
\end{equation}
A comparison of coefficients will lead to a contradiction. Let us
discuss identity \eqref{ex2} for fixed $r$. Since $\deg(p_1+p_2)=\max
\{\deg(p_1),\>\deg(p_2)\}$ for any two polynomials $p_1$, $p_2\in T$,
each summand in \eqref{ex2} has degree $\le3$. In particular, $\deg
(s_0)\le2$, $\deg(s_2)\le0$, and furthermore $\deg(g_i)=1$ and $t_i
\in R_\plus$ for each $i$. Writing $g_i=a_ix+b_iy+c_i$, the
coefficient of $x^3$ on the right is $d+\sum_it_ia_i^2$, where $d$
($\ge0$) is the coefficient of $x^2$ in $s_1$. Comparing with the
left hand side we conclude $d=0$ (hence $s_1\in R_\plus$) and $t_ia_i
=0$ for all $i$.

For any index $i$ with $t_i\ne0$, we have
$$h_i=b_i^2\,xy(x+y)+4b_ic_i\,xy+c_i^2\,(x+y).$$
Comparing coefficients of $x^2$ on both sides of \eqref{ex2}, we
conclude that $s_0$ must be a scalar as well, and thus even $s_0=0$.
Finally we compare coefficients of $x^2y$, $xy$ and $x$ in
\eqref{ex2}. This gives the identities
\begin{eqnarray*}
r^2 & = & \sum_it_ib_i^2, \\[-2pt]
-4r & = & s_2+4\sum_it_ib_ic_i, \\[-2pt]
1 & = & s_1+\sum_it_ic_i^2.
\end{eqnarray*}
Writing $u=\bigl(b_i\sqrt{t_i}\bigr)_i$ and $v=\bigl(c_i\sqrt{t_i}
\bigr)_i$ (two vectors in $R^m$) we conclude
\begin{equation}\Label{ineqs}
||u||^2=r^2,\quad\langle u,v\rangle\le-r,\quad||v||^2\le1.
\end{equation}
If $r>0$, the Cauchy-Schwarz inequality implies that both
inequalities in \eqref{ineqs} must be equalities. Therefore $u$ and
$v$ must be linearly dependent,
and hence $u=\pm rv$. There is an index $i$ with $t_i\ne0$, and we
conclude $r=\bigl|\frac{b_i}{c_i}\bigr|$.

But this shows that there are only finitely many values $r>0$ for
which an identity \eqref{ex2} is possible with the fixed choice of
generators. Hence we have proved that $\rey(T)$ is not finitely
generated.
\end{example}

\begin{lab}
If $T$ is a $G$-invariant preordering in $R[V]$, then $\rey(T)$ is
again a preordering, by Corollary \ref{reyinvmod}(c). Is this even
true if we drop the assumption that $T$ is $G$-invariant? Our third
example shows that the answer is no.
\end{lab}

\begin{example}\Label{reynotpo}
Consider once more the group $G$ of order two, acting on $V=\A^2$ by
$(x,y)\mapsto(y,x)$. Let $T$ be the preordering in $R[V]=R[x,y]$
generated by $g=1+x$ and $h=y^2+x$. We'll show that $\rey(g)\,\rey(h)
\notin\rey(T)$, which implies that $\rey(T)$ is not a preordering.
Suppose to the contrary that
\begin{equation}\Label{ex3}
\rey(g)\,\rey(h)=\rey(t)
\end{equation}
for some $t\in T$. There exist sums of squares $p$, $q$, $r$, $s$ in
$R[x,y]$ with $t=p+qg+rh+sgh$. Again we have $\deg(f_1+f_2)=
\max\{\deg(f_1),\>\deg(f_2)\}$ for any $f_1$, $f_2$ in the
preordering generated by $T$ and $T^\tau$.
Hence $\deg(p)$, $\deg(q)\le2$ and $r$, $s\in R_\plus$. Evaluating
both sides of \eqref{ex3} at the origin shows $p(0,0)=q(0,0)=0$, and
so $p$, $q$ are homogeneous of degree two.

Consider the point $M:=(-1,0)$ and its conjugate $M'=(0,-1)$ in the
$(x,y)$-plane. To evaluate both sides of \eqref{ex3} at $M$, we
record $g(M)=0$, $g(M')=1$, $h(M)=-1$ and $h(M')=1$. Hence $\rey(h)$
vanishes at $M$, and we get
\begin{eqnarray*}
0\ =\ \rey(t)(M) & = & \rey(p)(M)+\rey(qg)(M)+\rey(rh)(M)+\rey(sgh)
  (M) \\
& = & \frac12\bigl(p(M)+p(M')\bigr)+\frac12q(M')+\frac12s.
\end{eqnarray*}
It follows that $p(M)=p(M')=q(M')=s=0$. Since $p$ is a psd quadratic
form, we have $p=0$. Similarly, we get $q=ax^2$ with $a\ge0$, which
gives $2\rey(qg)=a(x^2+y^2+x^3+y^3)$. Re-writing \eqref{ex3} gives
$$(2+x+y)(x^2+y^2+x+y)=2a(x^2+y^2+x^3+y^3)+2r(x+y+x^2+y^2),$$
and comparing the coefficients of $xy$ we see a contradiction.
\end{example}


\section{Moment problems with symmetries}

In the second part of this paper we study moment problems on which a
group of symmetries acts. Therefore, our ground field will now always
be $R=\R$, the field of usual real numbers. Otherwise we'll keep the
situation considered so far. So we have the reductive group $G$ over
$\R$ which acts on the affine $\R$-variety $V$ via a morphism
$G\times V\to V$ of varieties. The quotient morphism is $\pi\colon V
\to V\qu G=W$, where $W$ is the affine variety with $\R[W]=\R[V]^G$
(and $\pi$ is induced by the inclusion $\R[V]^G\subset\R[V]$ of
rings). The image set of $\pi\colon V(\R)\to W(\R)$ is denoted $Z$.
By $\rey$ or $\rey_V$ we denote the Reynolds operator $\R[V]\to\R[W]$
(see \ref{defreynolds}).

\begin{lab}
Let $\R[V]^\du$ denote the dual vector space of $\R[V]$, i.~e.\ the
space of all linear functionals $L\colon\R[V]\to\R$. The right action
of $G(\R)$ on $\R[V]$ induces a left action on $\R[V]^\du$, namely
$(g,L)\mapsto\bigl({}^gL\colon f\mapsto L(f^g)\bigr)$. A linear
functional $L\in\R[V]^\du$ is called \emph{$G$-invariant} if $L$ is
invariant under this action of $G(\R)$, i.~e., if
$$L(f^g)=L(f)$$
holds for every $f\in\R[V]$ and $g\in G(\R)$.
\end{lab}

\begin{lem}\Label{ginvlinf}
A linear map $L\in\R[V]^\du$ is $G$-invariant if and only if $L(f)=
L(\rey f)$ for all $f\in\R[V]$. Hence the map
$$\rey^\du\colon\R[W]^\du\to\R[V]^\du,\quad F\mapsto F\comp\rey$$
is an isomorphism from $\R[W]^\du$ onto the space of $G$-invariant
linear forms on $\R[V]$.
\end{lem}

\begin{proof}
Consider the decomposition $\R[V]=\bigoplus_\omega\R[V]_{(\omega)}$
of the $G$-module $\R[V]$ into isotypical components, and let $\R[V]_
{(\omega_0)}:=\R[V]^G=\R[W]$ denote the submodule of $G$-invariants.
Any $G$-invariant linear form $L\colon\R[V]\to\R$ is a homomorphism
of $G$-modules, where $\R$ is given the trivial $G$-module structure.
But $\Hom_G(M,\R)=\{0\}$ for any irreducible $G$-module $M\ne\R$.
This shows that any $G$-invariant linear form $L$ vanishes on each
$\R[V]_{(\omega)}$, $\omega\ne\omega_0$. Since $\rey$ is the
projection of $\R[V]$ onto $\R[V]_{(\omega_0)}=\R[W]$, the lemma
follows from this.
\end{proof}

\begin{lab}
We are going to relate (Borel) measures on $W(\R)$ to $G$-invariant
(Borel) measures on $V(\R)$. We only have results when $G(\R)$ is
compact. Given a measure $\mu$ on $V(\R)$ for which every $f\in\R[V]$
is $\mu$-integrable, we write $L_\mu\colon f\mapsto\int_{V(\R)}f\>
d\mu$ for the linear functional ``integration by~$\mu$''.

A measure $\mu$ on $V(\R)$ will be called \emph{$G$-invariant} if
$\mu$ is invariant under translation by all elements of $G(\R)$,
i.~e., if $\mu(gA)=\mu(A)$ for every Borel set $A\subset V(\R)$ and
every $g\in G(\R)$.
Note that this implies
$$\int f\,d\mu=\int f^g\,d\mu$$
for every $f\in\R[V]$ and $g\in G(\R)$ (provided that either side
exists).
As usual, if $f\colon X\to Y$ is any continuous map of topological
spaces and $\mu$ is a Borel measure on $X$, then $f_*(\mu)$ denotes
the direct image measure on $Y$. It is characterized by $(f_*\mu)(B)=
\mu(f^{-1}(B))$ ($B\subset Y$ any Borel set).
\end{lab}

We first recall a few basic relations between invariant measures on a
$G$-space and measures on the orbit spaces. These facts must
certainly be folklore among the experts. Since we have not been able
to find suitable references, we decided to include the (easy) proofs.
For the rest of this section, assume that the group $G(\R)$ is
compact.

\begin{lem}\Label{charinvmeas}
Let $\sigma\colon G(\R)\times V(\R)\to V(\R)$ be the group action. A
Borel measure $\mu$ on $V(\R)$ is $G$-invariant if and only if
$$\mu=\sigma_*(\lambda\otimes\mu),$$
where $\lambda$ is the normalized Haar measure on $G(\R)$.
\end{lem}

\begin{proof}
For $g\in G(\R)$ let $l_g\colon x\mapsto gx$ denote left translation
by $g$ on either $V(\R)$ or $G(\R)$. If $\mu=\sigma_*(\lambda\otimes
\mu)$ then $\mu$ is $G$-invariant since
$$(l_g)_*(\mu)=(l_g\comp\sigma)_*(\lambda\otimes\mu)=(\sigma\comp
(l_g\times\id))_*(\lambda\otimes\mu)=\sigma_*(\lambda\otimes\mu)=
\mu.$$
Conversely, if $\mu$ is $G$-invariant, then for each Borel set $B$ in
$V(\R)$ we have
\begin{eqnarray*}
\sigma_*(\lambda\otimes\mu)(B) & = & (\lambda\otimes\mu)\bigl(
  \sigma^{-1}(B)\bigr) \\
& = & \int_{G(\R)}\mu(g^{-1}B)\>\lambda(dg)=\int_{G(\R)}\mu(B)\>
\lambda(dg)=\mu(B)
\end{eqnarray*}
by the Fubini formula.
\end{proof}

\begin{prop}\Label{measwginvmeasv}
Let $\nu$ be any Borel measure on $Z=\im(\pi)\subset W(\R)$.
\begin{itemize}
\item[(a)]
There exists a unique $G$-invariant Borel measure $\mu$ on $V(\R)$
with $\pi_*(\mu)=\nu$. We will denote it by $\mu=:\pi^*(\nu)$.
\item[(b)]
Explicitly, if $f\colon V(\R)\to\R$ is a non-negative measurable
function, then
$$\int_{V(\R)}f(x)\>\mu(dx)\>=\>\int_Z\ol h_f(y)\>\nu(dy),$$
where the function $\ol h_f\colon Z\to\R\cup\{\infty\}$ is defined by
$$\ol h_f(\pi x)\>:=\>\int_{G(\R)}f(gx)\>\lambda(dg)$$
($x\in V(\R)$).
\end{itemize}
\end{prop}

\begin{proof}
Let $\lambda$ be the normalized Haar measure on $G(\R)$. We will need
that $\lambda$ is invariant under left and right translation by
elements of $G(\R)$.

Given a Borel set $A$ in $V(\R)$, let $x\in V(\R)$, and write $o_x
\colon G(\R)\to V(\R)$ for the orbit map $o_x(g):=gx$. Put
$$h_A(x)\>:=\>(o_{x*}\lambda)(A)\>=\>\lambda\,\bigl\{g\in G(\R)\colon
gx\in A\bigr\}.$$
The function $h_A\colon V(\R)\to[0,1]$ is measurable (\cite{Ba}
23.2).
Since $o_{gx}=o_x\comp r_g$ (where $r_g\colon h\mapsto hg$ is right
translation by $g\in G(\R)$) we have
$$h_A(gx)=(o_{x*}r_{g*}\lambda)(A)=(o_{x*}\lambda)(A)=h_A(x)$$
for $x\in V(\R)$ and $g\in G(\R)$.
Since the fibres of $\pi\colon V(\R)\onto Z$ are precisely the
$G(\R)$-orbits (Theorem \ref{procesischw}(a)), $h_A$ induces a
measurable function $\ol h_A\colon Z\to[0,1]$ by $\ol h_A(\pi(x))=h_A
(x)$ ($x\in V(\R)$).

Now let $\nu$ be a Borel measure on $Z$, and define
$$\mu(A):=\int_Z\ol h_A(y)\>\nu(dy)$$
for $A\subset V(\R)$ a Borel set. Then $\mu$ is a Borel measure on
$V(\R)$, since if $A=\bigcup_{n\in\N}A_n$ is a countable union of
pairwise disjoint Borel sets in $V(\R)$, we have $h_A=\sum_nh_{A_n}$
(pointwise on $V(\R)$), and therefore $\mu(A)=\sum_n\mu(A_n)$. Also
$h_{gA}=h_A$ for $g\in G(\R)$, and so $\mu$ is $G$-invariant.
>From the construction it is clear that
$\pi_*(\mu)=\nu$.

Let $f\colon V(\R)\to\R_\plus$ be measurable. The function $h_f\colon
V(\R)\to\R\cup\{\infty\}$, $h_f(x):=\int_{G(\R)}f(gx)\>\lambda(dg)$
is again measurable and $G$-invariant, so it induces a measurable
function $\ol h_f\colon Z\to\R\cup\{\infty\}$ as in the proposition.
Given any $G$-invariant measure $\tilde\mu$ on $V(\R)$ with $\pi_*
(\tilde\mu)=\nu$, we have
\begin{eqnarray*}
\int_{V(\R)}f(x)\>\tilde\mu(dx) & = & \int_{V(\R)}\int_{G(\R)}f(gx)\>
  \lambda(dg)\>\tilde\mu(dx) \\
& = & \int_{V(\R)}\ol h_f(\pi(x))\>\mu(dx) \\
& = & \int_Z\ol h_f(y)\>\pi_*(\mu)(dy) \\
& = & \int_Z\ol h_f(y)\>\nu(dy) \\
\end{eqnarray*}
using $\tilde\mu=\sigma_*(\lambda\otimes\tilde\mu)$ and Fubini's
theorem. This establishes both the uniqueness of $\mu$ and the second
part of the proposition.
\end{proof}

\begin{cor}
The operators $\pi_*$ and $\pi^*$ set up a bijective correspondence
between the set of $G$-invariant Borel measures $\mu$ on $V(\R)$ and
the set of all Borel measures $\nu$ on $Z$. In particular, one has
$$\pi^*\pi_*(\mu)=\mu\text{ \ and \ }\pi_*\pi^*(\nu)=\nu.
\eqno\square$$
\end{cor}

\begin{cor}\Label{munurelation}
Let $\nu$ be a Borel measure on $Z$, and let $\mu=\pi^*\nu$.
\begin{itemize}
\item[(a)]
For any $f\in\R[V]$ we have $\int_{V(\R)}f\>d\mu=\int_Z\rey(f)\>
d\nu$ (one integral exists iff the other exists).
\item[(b)]
For any $G(\R)$-invariant Borel set $A$ in $V(\R)$ we have $\mu(A)=
\nu(\pi(A))$.
\item[(c)]
Given any $G$-invariant closed subset $K$ of $V(\R)$, we have $\supp
(\mu)\subset K$ iff $\supp(\nu)\subset\pi(K)$.
\end{itemize}
\end{cor}

\begin{proof}
(b) is clear from $\nu=\pi_*(\mu)$ (using again that the fibres of
$\pi$ are the $G(\R)$-orbits), and (c) follows from (b). As to (a),
we have
$$\int_{V(\R)}f\>d\mu=\int_Z\ol h_f(y)\>d\nu(y)$$
by \ref{measwginvmeasv}. But $\ol h_f(\pi x)=\int_{G(\R)}f(gx)\>
d\lambda(g)$ is just $(\rey f)(x)$, for $x\in V(\R)$, see Proposition
\ref{reyhaar}. So $\ol h_f=\rey(f)$ as functions on $Z$.
\end{proof}

\begin{rem}
In the situation of Corollary \ref{munurelation}, assume that every
function in $\R[V]$ is $\mu$-integrable (or equivalently, by
\ref{munurelation}(a), that every function in $\R[W]$ is
$\nu$-integrable). Then the linear forms $L_\mu\in\R[V]^\du$ and
$L_\nu\in\R[W]^\du$ are related by
\begin{center}
$L_\mu=L_\nu\comp\rey$ \ and \ $L_\nu=(L_\mu)\big|_{\R[W]}$.
\end{center}
\end{rem}

\begin{lem}\Label{equiconds}
Let $K\subset V(\R)$ be a $G$-invariant basic closed set, and let
$L\in\R[V]^\du$ be a $G$-invariant linear form. The following
conditions are all equivalent:
\begin{itemize}
\item[(i)]
There is a measure $\mu$ on $K$ with $L(f)=\int_Kf\>d\mu$ for every
$f\in\R[V]^G$;
\item[(ii)]
there is a $G$-invariant measure $\mu$ on $K$ with $L(f)=\int_Kf\>
d\mu$ for every $f\in\R[V]$;
\item[(iii)]
there is a measure $\nu$ on $\pi(K)$ with $L(f)=\int_{\pi(K)}f\>d\nu$
for every $f\in\R[W]$;
\item[(iv)]
$L(f)\ge0$ for every $f\in\scrP_V(K)^G=\scrP_W(\pi K)$.
\end{itemize}
\end{lem}

\begin{proof}
(i) $\To$ (iii):
Assuming (i), let $\nu:=\pi_*(\mu)$. Then for every $f\in\R[W]$ we
have $\int_{\pi(K)}f\>d\nu=\int_Kf\comp\pi\>d\mu=L(f)$ by (i).
(iii) $\To$ (ii):
Assuming (iii), put $\mu:=\pi^*(\nu)$. Then $\mu$ is $G$-invariant.
Using \ref{munurelation}(a) we have $\int_Kf\>d\mu=\int_{\pi(K)}\rey
(f)\>d\nu=L(\rey(f))=L(f)$ for $f\in\R[V]$, where the last equality
holds since $L$ is $G$-invariant (\ref{ginvlinf}). The implication
(ii) $\To$ (i) is trivial.

By Haviland's theorem, (iii) is equivalent to $L(f)\ge0$ for every
$f\in\scrP_W(\pi K)$. By Corollary \ref{basic}, $\scrP_W (\pi K)=
\scrP_V(K)^G$.
\end{proof}

\begin{cor}
Let $L\in\R[V]^\du$ be a linear functional on $\R[V]$ which is
integration with respect to some Borel measure on $V(\R)$.
Then $L$ is $G$-invariant if and only if there exists a $G$-invariant
Borel measure $\mu$ on $V(\R)$ with $L=L_\mu$.
\end{cor}

\begin{proof}
The `if' part is obvious anyway. Conversely, if $\tilde\mu$ is some
measure on $V(\R)$ with $L=L_{\tilde\mu}$, then $\mu:=\pi^*\pi_*
\tilde\mu$ is a $G$-invariant Borel measure with $L_\mu=L$ (see (i)
$\To$ (iii) $\To$ (ii) in the last proof).
\end{proof}

\begin{lab}\Label{recallkmp}
Given a basic closed set $K\subset V(\R)$ and a quadratic module $M$
in $\R[V]$, recall that $M$ is said to \emph{solve the $K$-moment
problem} if the linear forms $L\in\R[V]^\du$ with $L|_M\ge0$ are
precisely the $K$-moment functionals, i.e., the linear forms
represented by Borel measures on $K$. By Haviland's theorem, it is
equivalent that the closure $\ol M$ of $M$ is equal to $\scrP_V(K)$.
Here and in the sequel, the closure refers to the finest locally
convex vector space topology on $\R[V]$. See \cite{PSch}, for
example.

If, for given $K$, such $M$ can be found which is finitely generated
(as a quadratic module), this allows a characterization of the
$K$-moment functionals by an explicit recursive sequence of
conditions. Indeed, if $M=\Sigma f_1+\cdots+\Sigma f_r$, say (with
$\Sigma:=\Sigma\R[V]^2$), then $L$ is a $K$-moment functional if and
only if $L(q^2f_i)\ge0$ for every $q\in\R[V]$ and $i=1,\dots,r$; and
for a fixed $i$, this translates into a positive semidefiniteness
condition for a countable generalized Hankel matrix which depends in
a direct explicit way on $L$ and $f_i$.

Following \cite{Sm}, we say that $M$ has the \emph{strong moment
property (SMP)} if the closure $\ol M$ of $M$ is saturated. Thus $M$
solves the $K$-moment problem iff $M$ has (SMP) and $\scrX_V(M)=
\wt K$.

We are now going to study variants of this notion which take the
group action into account. The idea is, while it may be hard or even
impossible to characterize all moment functionals of Borel measures
on $K$, the task may become easier if one only aims at characterizing
the \emph{invariant} moment functionals.
\end{lab}

\begin{dfn}
Let $K$ be a $G$-invariant basic closed set in $V(\R)$, and let $N$
be a
quadratic module in $\R[W]=\R[V]^G$. We'll say that $N$ \emph{solves
the invariant $K$-moment problem}, if the following is true for every
$G$-invariant linear functional $L\colon\R[V]\to\R$:
\begin{quote}
\emph{Conditions (i)--(iv) of \ref{equiconds} hold for $L$ if and
only if $L(f)\ge0$ for every $f\in N$.}
\end{quote}
\end{dfn}

\begin{cor}\Label{charactimp}
Let $K\subset V(\R)$ be basic closed and $G$-invariant, and let $N$
be a
quadratic module in $\R[W]$. The following are equivalent:
\begin{itemize}
\item[(i)]
$N$ solves the invariant $K$-moment problem;
\item[(ii)]
$\ol N=\scrP_W(\pi K)$;
\item[(iii)]
$N$ solves the (usual) $\pi(K)$-moment problem in $W$.
\end{itemize}
\end{cor}

\begin{proof}
The equivalence of (ii) and (iii) is well-known (and recalled in
\ref{recallkmp}). The equivalence of (i) and (iii) is clear from
Lemma \ref{equiconds}.
\end{proof}

\begin{lab}
We also introduce a weakening of this notion. Let again $K$ be a
$G$-invariant basic closed set in $V(\R)$, and let $M$ be a
quadratic module in $\R[V]$. We say that $M$ solves the
\emph{averaged $K$-moment problem} if the following is true for every
$G$-invariant linear functional $L\colon\R[V]\to\R$:
\begin{quote}
\emph{Conditions (i)--(iv) of \ref{equiconds} hold for $L$ if and
only if $L(f)\ge0$ for every $f\in M$.}
\end{quote}
\end{lab}

\begin{cor}\Label{charactamp}
Let $K\subset V(\R)$ be basic closed and $G$-invariant, and let $M$
be a quadratic module in $\R[V]$. The following are equivalent:
\begin{itemize}
\item[(i)]
$M$ solves the averaged $K$-moment problem;
\item[(ii)]
$\ol{\rey(M)}=\scrP_W(\pi K)$;
\item[(iii)]
$\rey(M)$ solves the (usual) $\pi(K)$-moment problem in $W$.
\end{itemize}
\end{cor}

\begin{proof}
(i) $\iff$ (ii) is clear from Lemma \ref{equiconds}. Again, (ii)
$\iff$ (iii) is well-known (see \ref{recallkmp}).
\end{proof}

\begin{rem}
Let $K\subset V(\R)$ be $G$-invariant, and let $N$ be a quadratic
module in $\R[W]$. If $N$ solves the invariant $K$-moment problem,
then $K$ is determined by $N$ via $K=\scrS_V(N)$. Similarly, if $M
\subset\R[V]$ solves the averaged $K$-moment problem, then $K=\scrS_V
(M)$.

Keeping this in mind, Corollary \ref{charactimp} justifies the
following terminology: A quadratic module $N$ in $\R[W]$ has the
\emph{invariant moment property (IMP)} if $\ol N$ is saturated in
$\R[W]$, that is, if $N$ has the (SMP) (in $\R[W]$).

Similarly, Corollary \ref{charactamp} justifies us to say: A
quadratic module $M$ in $\R[V]$ has the \emph{averaged moment
property (AMP)} if $\ol{\rey(M)}$ is saturated in $\R[W]$, that is,
if $\rey(M)$ has the (SMP) (in $\R[W]$).

Thus, $M$ has the (AMP) if and only if $\rey(M)$ has the (IMP). Of
course, we are usually only interested in (IMP) or (AMP) when the
corresponding module is finitely generated.
\end{rem}

\begin{prop}\Label{mpimp}
Let $K\subset V(\R)$ be a $G$-invariant basic closed set, and let $M$
be a quadratic module in $\R[V]$ which solves the (usual) $K$-moment
problem in $\R[V]$. Then $M$ has the averaged moment property (AMP).
\end{prop}

\begin{proof}
The hypothesis says $\ol M=\scrP_V(K)$. Application of the Reynolds
operator yields
\begin{center}
$\rey(M)\subset\rey(\scrP_V(K))=\rey(\ol M)\subset\ol{\rey(M)}$.
\end{center}
Since $\rey(\scrP_V(K))=\scrP_W(\pi K)$ (\ref{basic}) is closed,
this implies $\ol{\rey(M)}=\scrP_W(\pi K)$.
\end{proof}

\begin{rem}\Label{modimpamp}
We now compare the various moment properties considered. Let $K$
be a $G$-invariant basic closed set in $V(\R)$, let $N$ be a
quadratic module in $\R[V]^G=\R[W]$, and let $M$ be the quadratic
module in $\R[V]$ generated by $N$. We consider the following three
properties of $N$:
\begin{itemize}
\item[(IMP$_K$)]
$N$ solves the invariant $K$-moment problem;
\item[(AMP$_K$)]
$M$ solves the averaged $K$-moment problem;
\item[(SMP$_K$)]
$M$ solves the (ordinary) $K$-moment problem (in $V$).
\end{itemize}
\end{rem}

\begin{cor}\Label{remark 6.18}
We have the implications (SMP$_K$) $\To$ (AMP$_K$) and (IMP$_K$)
$\To$ (AMP$_K$). On the other hand, neither (IMP$_K$) nor (AMP$_K$)
implies (SMP$_K$), in general.
\end{cor}

\begin{proof}
The first implication follows by Proposition \ref{mpimp}. The second
implication holds since $\ol N=\scrP_W(\pi K)$ and $N\subset\rey(M)
\subset\scrP_W(\pi K)$ (Lemma \ref{lem2po}) imply $\ol{\rey(M)}=
\scrP_W(\pi K)$.
Examples \ref{easyexample} and \ref{salmabsp} show that (IMP$_K$)
does not imply (SMP$_K$). Therefore, (AMP$_K$) does not imply
(SMP$_K$) either.
\end{proof}

We were not able to decide whether (SMP$_K$) implies (IMP$_K$), nor
whether (AMP$_K$) implies (IMP$_K$). See Open Question~1.

\begin{rem}\Label{mpimpamp}
In a similar vein, one may consider the following three properties of
the $G$-invariant basic closed set $K$:
\begin{itemize}
\item[(IMP)]
The invariant $K$-moment problem is finitely solvable;
\item[(AMP)]
the average $K$-moment problem is finitely solvable;
\item[(SMP)]
the (ordinary) $K$-moment problem is finitely solvable.
\end{itemize}
Of course, ``finitely solvable'' means solvable by finitely generated
quadratic modules, in $\R[V]$ for (AMP) and (SMP), and in $\R[W]$ for
(IMP).

Clearly, (SMP) $\To$ (AMP) and (IMP) $\To$ (AMP). Examples
\ref{easyexample} and \ref{salmabsp} show (IMP) $\not\To$ (SMP) (and
therefore (AMP) $\not\To$ (SMP)). Example \ref{timbsp} (due to Tim
Netzer) shows that (SMP) $\not\To$ (IMP) (and therefore (AMP) $\not
\To$ (IMP)).
\end{rem}

\begin{rem}\Label{themoral}
For describing the \emph{$G$-invariant} linear functionals $L\in\R[V]
^\du$ which correspond to measures on $K$ (or equivalently, to
$G$-invariant such measures, see \ref{equiconds}), (IMP) and (AMP)
are in principle as good as (SMP). Indeed, (AMP) and (SMP) both mean
that one has to check $L(q^2f_i)\ge0$ for all $q\in\R[V]$ and a
finite number of fixed $f_i\in\R[V]$. For (IMP) it is the same,
except that the $f_i$ are in addition $G$-invariant and one may
restrict to $G$-invariant multipliers $q$, which further simplifies
the task. Therefore, from a practical point of view, it is
interesting to find situations where (IMP) or (AMP) holds, but (SMP)
fails.

Such situations are constructed in \ref{easyexample} and
\ref{salmabsp}: Here (SMP) fails, while (IMP) (and thus (AMP)) holds.
\end{rem}

\begin{rem}\Label{mpsolvbyinvfcts}
If $K$ is a $G$-invariant basic closed set in $V(\R)$, then the
$K$-moment problem need not be solvable by any number of
\emph{$G$-invariant} functions.

In other words, what we are claiming is that the preordering $T$ in
$\R[V]$ generated by $\scrP_V(K)^G=\scrP_W(\pi K)$ need not be dense
in $\scrP_V(K)$, although the saturation of $T$ is clearly equal to
$\scrP_V(K)$.

Indeed, in Examples \ref{easyexample} or \ref{salmabsp} there exists
a finitely generated preordering $N$ in $\R[W]$ satisfying $\ol N=
\scrP_W(\pi K)$. Let $M$ be the preordering generated by $N$ in
$\R[V]$. Then $\scrP_W(\pi K)=\ol N\subset\ol M$, and therefore also
$\ol T\subset\ol M$. On the other hand, $\ol M\subsetneq\scrP_V(K)$,
since $M$ is finitely generated but the $K$-moment problem is not
finitely solvable in these examples.

Other classes of examples illustrating the point of this remark will
be constructed in \cite{CKM}.
\end{rem}

Our last result is of a negative character. It provides a large class
of cases where the ordinary $K$-moment problem is not finitely
solvable, and where the invariant $K$-moment problem is not finitely
solvable either:

\begin{thm}\Label{conevw}
Let the finite group $G$ act on the irreducible normal affine
$\R$-variety $V$, and let $K$ be a $G$-invariant basic closed set in
$V(\R)$. Suppose that there exists a $G$-equivariant completion
$V\into X$ of $V$ such that
\begin{itemize}
\item[(1)]
$X$ is normal and projective,
\item[(2)]
for every irreducible component $X'$ of $X-V$, $\ol K\cap X'(\R)$ is
Zariski dense in $X'$.
\end{itemize}
Let $W=V\qu G$. Then every finitely generated quadratic module $N$ in
$\R[W]$ with $\scrS_W(N)=\pi(K)$ is stable and closed. In particular,
if $\dim(V)\ge2$, then the $G$-invariant $K$-moment problem is not
finitely solvable.
\end{thm}

\begin{proof}
Note that in (2), $\ol K$ denotes the closure of $K$ inside $X(\R)$.
The variety $X$ has a covering by open affine $G$-invariant subsets,
since every $G$-orbit is contained in some open affine set $U\subset
X$, and since the subset $\bigcap_{g\in G}gU$ is affine and
$G$-invariant. Therefore one can form the (geometric) quotient
variety $Y=X/G$. The morphism $\pi\colon X\to Y$ is finite, and $Y$
is a normal and complete variety (\cite{Kr} p.~100). The inclusion
$V\into X$ induces an open embedding $W\into Y$. The square of
morphisms
$$\xymatrix{
V \ar[r] \ar[d] & X \ar[d]^\pi \\
W \ar[r] & Y}$$
is cartesian.

Write $X_0:=X-V$ and $Y_0:=Y-W$. Every irreducible component $Y'$ of
$Y_0$ has the form $Y'=\pi(X')$ with some irreducible component $X'$
of $X_0$. We claim that $Y'(\R)\cap\ol{\pi(K)}$ (the closure taken in
$Y(\R)$) is Zariski dense in $Y'$.

Indeed, by hypothesis (2) there exist $\beta\in\wt K$ and $\alpha\in
\wt{X'(\R)}$ with support of $\alpha$ equal to $X'$ and with $\beta
\succ\alpha$ (specialization in $\wt{X(\R)}$, the real spectrum of
$X$). Applying the map $\wt\pi\colon\wt{X(\R)}\to\wt{Y(\R)}$ of real
spectra we get $\wt\pi(\beta)\succ\wt\pi(\alpha)$. Moreover, $\wt\pi
(\beta)\in\wt{\pi(K)}$, and the support of $\wt\pi(\alpha)$ is
$\pi(X')=Y'$, which proves the claim.

The affine variety $W=V\qu G$ is again normal. So we can apply
\cite{PSch} Theorem 2.14 to $W$ and its completion $Y$, and to the
basic closed subset $\pi(K)$ of $W(\R)$. By this result, every
quadratic module $N$ in $\R[W]$ with $\scrS_W(N)=\pi(K)$
is stable and closed. Since $\dim(K)=\dim(V)$ (as follows from (2)),
$N$ cannot have the (SMP) if $\dim(V)\ge2$, by \cite{Sch:stable}
Thm.\ 5.4. By Corollary \ref{charactimp}, this means that $N$ cannot
solve the invariant $K$-moment problem.
\end{proof}

\begin{example}\Label{exconevw}
The proposition applies in particular when $V$ is a linear
representation space of $G$ and the set $K$ contains a non-empty open
cone. Indeed, every linear $G$-action on $V=\A^n$ extends to a linear
$G$-action on $X=\P^n$.
\end{example}


\section{More examples}

Let $K$ be a basic closed \sa\ set in $V(\R)$ which is $G$-invariant.
We consider the finite solvability of the $K$-moment problem on the
one hand, and of the $G$-invariant $K$-moment problem on the other.
In general, the question of characterizing the solutions of the
moment problem will become easier when one restricts attention to
solutions with symmetries. Our first examples are meant to
demonstrate this fact. They show that finite characterizations of the
$G$-invariant solutions may be available at the same time when such
characterizations do not exist for the class of all solutions.

Here is a first class of examples in dimension one.

\begin{lab}\Label{easyexample}
Let $f(x)$ be a square-free monic polynomial in $\R[x]$ of degree
$d$, and consider the set $K=\{(x,y)\in\R^2\colon y^2=f(x)\}$. The
group $G$ of order two acts on $K$ by $(x,y)\mapsto(x,-y)$. The
usual $K$-moment problem fails to be finitely solvable if $d\ge3$,
whereas the $G$-invariant $K$-moment problem can be finitely solved.

Indeed, $K$ is the set of $\R$-points of an affine non-singular
hyperelliptic curve $C$, whose genus $g=\lfloor\frac{d-1}2\rfloor$ is
positive for $d\ge3$. All points of $C$ at infinity are real (there
are one or two of them, according to whether $d$ is odd or even).
Hence the $K$-moment problem is not finitely solvable (see
\cite{PSch} Corollary 3.10).

On the other side, the quotient curve $C\qu G$ is the affine line,
and $\pi\colon C\to C\qu G=\A^1$ is the map $\pi(x,y)=x$. Hence $Z=
\{x\in\R\colon f(x)\ge0\}$, and the $Z$-moment problem is solved by
a finite number of polynomials in $x$, which are explicit in terms of
the real zeros of $f(x)$ (see \cite{Sch:MZ} 5.23.1 or \cite{KM} 2.2).
So these polynomials solve the $G$-invariant $K$-moment problem.

Many variations of this example can be built, using proper
$G$-invariant closed subsets $K$ of $C(\R)$.
\end{lab}

\begin{lab}\Label{salmabsp}
For a two-dimensional example consider the dihedral group $G=D_4$ of
order eight acting on the real affine plane $V=\A^2$ in the natural
way (as the symmetry group of a square centered at the origin). The
basic closed set
$$K:=\bigl\{(x,y)\in\R^2\colon-1\le(x^2-1)(y^2-1)\le0\bigr\}$$
in the plane $\R^2=V(\R)$ is $G$-invariant.
\begin{center}
\includegraphics*[height=55mm]{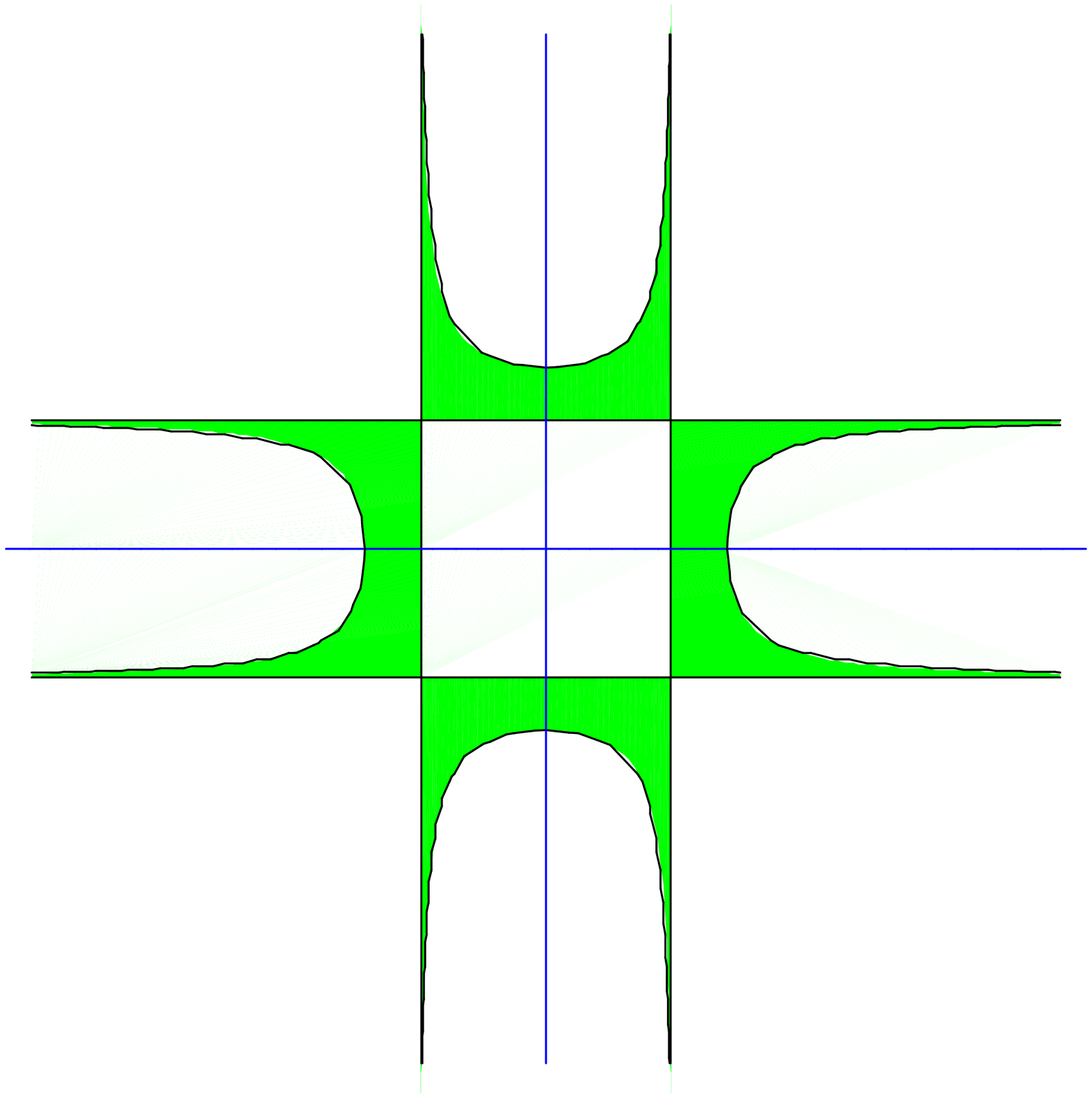}
\end{center}
The ring of invariants is $\R[x,y]^G=\R[u,v]$ with
$$u=x^2+y^2,\quad v=x^2y^2,$$
and $W=\A^2\qu G$ is itself an affine plane (see \cite{St}).
The image of $\pi\colon V(\R)\to W(\R)$ is
$$Z=\pi(\R^2)=\bigl\{(u,v)\in\R^2\colon u\ge0,\ v\ge0,\ u^2\ge4v
\bigr\}.$$
Since $(x^2-1)(y^2-1)=v-u+1$, we have
$$\pi(K)=\bigl\{(u,v)\in\R^2\colon v\ge0,\ 1\le u-v\le2\bigr\}.$$
This is a (half-) strip in the $(u,v)$-plane:
\begin{center}
\includegraphics*[height=55mm]{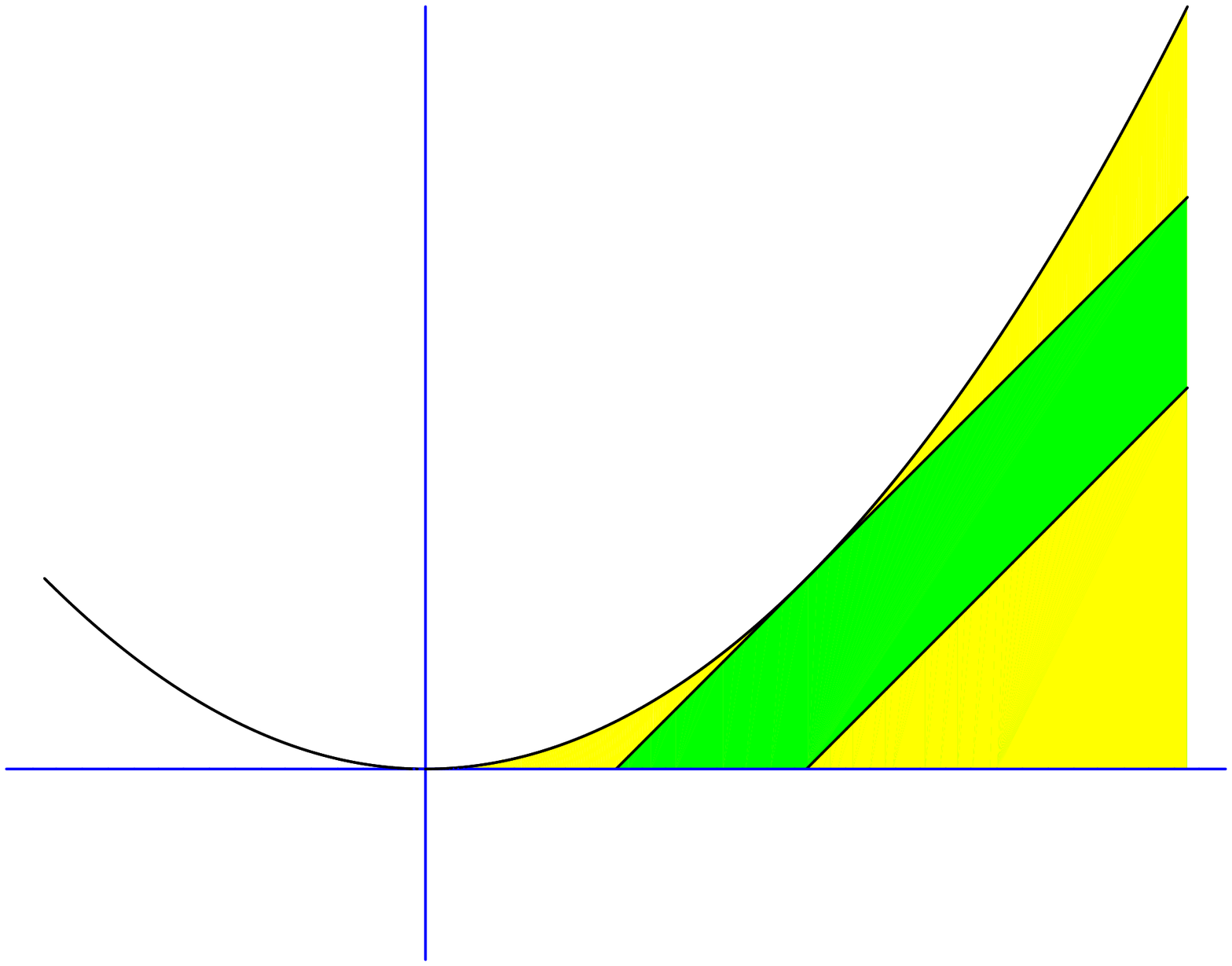}
\end{center}
The moment problem for $\pi(K)$ is solved by the preordering $N$ in
$\R[u,v]=W(\R)$ generated by $v$, $u-v-1$ and $2-u+v$ (see \cite{KMS}
Corollary 5.2). This means that the $G$-invariant $K$-moment problem
is solved by $N$ (Corollary \ref{charactimp}).

On the other hand, we'll show now that the usual $K$-moment problem
is not finitely solvable. Given a real parameter $c$, consider the
affine plane curve
$$E_c:\ (x^2-1)(y^2-1)=c.$$
Let $\P^2=\A^2\cup L$ be the projective plane, where $L$ is the line
at infinity. An easy calculation shows that the Zariski closure
$\ol E_c$ of $E_c$ in $\P^2$ meets $L$ in two points $P$ and $Q$,
both real, and furthermore $P$ and $Q$ are ordinary double points of
$\ol E_c$ (with real tangents).

Moreover, the affine curve $E_c$ is non-singular for $c\ne0,\>1$.
For these values, therefore, $E_c$ is a non-singular affine curve
of genus one
which has four points at infinity, all of them real.
Since $E_c(\R)\subset K$ for $-1\le c<0$, we conclude that the
$K$-moment problem is not finitely solvable, using \cite{PSch}
Corollary~3.10. (One single such value $c$ is already enough for the
argument.)
\end{lab}

We are grateful to Tim Netzer for finding the following example, and
allowing us to include it here.

\begin{lab}\Label{timbsp}
The following two dimensional example shows that (SMP) does not imply
(IMP). In particular, (AMP) does not imply (IMP).

Let the group $G$ of order two act on the affine plane $V=\A^2$ by
permuting the coordinates $x$ and $y$. The basic closed set
$$K:=\bigl\{(x,y)\in\R^2\colon x\ge0,\ y\ge0,\ xy\le1\bigr\}$$
is $G$-invariant. Furthermore, the preordering generated by $x$, $y$,
$1-xy$ in $\R(V)=\R[x,y]$ solves (SMP) for $K$, see for example
\cite{KMS}, Example 8.4.

The ring of $G$-invariant polynomials $\R[x,y]^G$ is a polynomial
ring $\R[u,v]$, where $u=x+y$ and $v=xy$. So $W:=\A^2\qu G$ is again
an affine plane, and the image of $\pi\colon V(\R)\to W(\R)$ is
$$Z=\pi(\mathbb{R}^2)=\bigl\{(u,v)\in\R^2\colon u^2\ge4v\bigr\}.$$
One checks that
$$\pi(K)=\bigl\{(u,v)\in\R^2\colon0\le u,\ 0\le v\le1,\ u^2\ge4v
\bigr\}$$
holds.
\begin{center}
\includegraphics*[height=55mm]{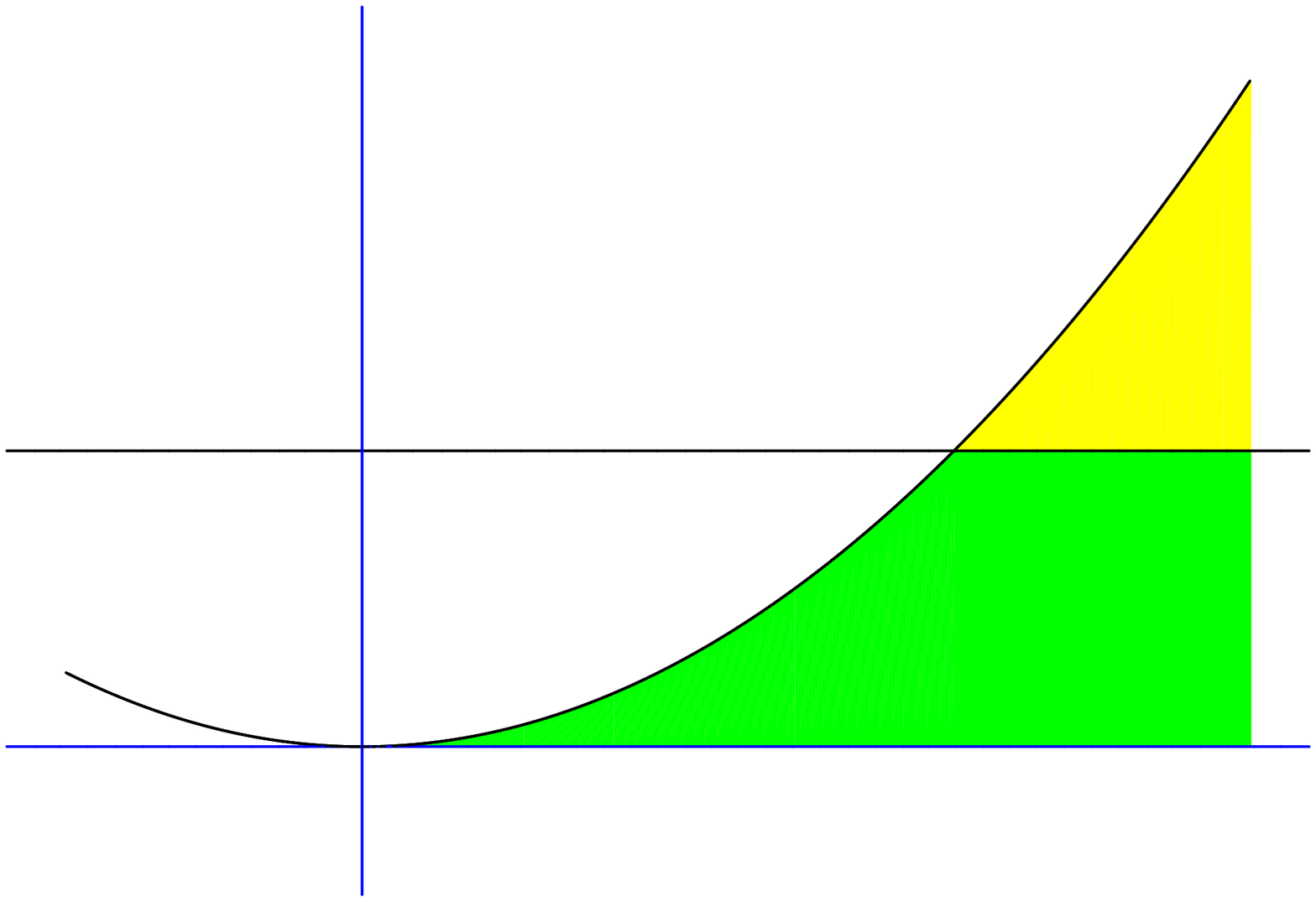}
\end{center}
\end{lab}

\begin{lem}
The $\pi(K)$-moment problem is not finitely solvable.
\end{lem}

\begin{proof}
Suppose there are polynomials $f_{1},\dots,f_{s}\in\R[u,v]$ such that
the preordering $PO(f_1,\dots,f_s)$ solves (SMP) for $\pi(K)$. Then
for any $b\in[0,1]$, the preordering
$$PO(f_1(u,b),\dots,f_s(u,b))\subset\R[u]$$
solves (SMP) for the set
$$\bigl[2\sqrt b,\,\infty\bigr[\subset\R,$$
by \cite{Sch:stable}, Prop.\ 4.8. By \cite{KM}, Theorems 2.1 and 2.2,
the natural generator for this set, namely $u-2\sqrt b$, must be
among the $f_i(u,b)$ up to a constant factor. So without loss of
generality, assume
$$f_1(u,b)=r(b)\,\bigl(u-2\sqrt b\bigr)$$
for infinitely many $b\in[0,1]$ and some positive function $r$.
Writing
$$f_1(u,v)=\sum_jg_j(v)\,u^j$$
and comparing coefficients, we get $g_0(b)=-2r(b)\sqrt b$ and
$g_1(b)=r(b)$ for infinitely many $b\in [0,1]$. So for all these
$b$,
$$g_0(b)^2=4r(b)^2b=4g_1(b)^2b,$$
so $g_0^2=4g_1^2v$ in $\R[v]$. As the left hand side has even and the
right hand side has odd degree, this is a contradition.
\end{proof}

\begin{rem}
Note that the example does not give a negative answer to the question
whether (SMP$_K$) implies (IMP$_K$). For this one would need to have
a collection of finitely many \emph{$G$-invariant} polynomials which
solve the $K$-moment problem. It can be shown that such a collection
does not exist.

So the question whether (SMP$_K$) implies (IMP$_K$) remains open.
\end{rem}


\section{Open questions}

1.\
Let $K\subset V(\R)$ be basic closed and $G$-invariant. Assume that
$N$ is a quadratic module in $\R[V]^G=\R[W]$, and $M$ is the
quadratic module generated by $N$ in $\R[V]$. If $M$ solves the
$K$-moment problem (on $V$), does it follow that $N$ solves the
$\pi(K)$-moment problem (on $W$)? In other words, does the
implication (SMP$_K$) $\To$ (IMP$_K$) hold (c.f.\ Remark
\ref{modimpamp} and Corollary \ref{remark 6.18})? In all cases where
(SMP$_K$) is satisfied and that we could analyze, it turned out that
(IMP$_K$) holds as well. Similarly, we do not know any example where
(AMP$_K$) holds and (IMP$_K$) fails.

2.\
What about the averaged moment property (AMP) in the situations
covered by Theorem \ref{conevw}, in particular Example
\ref{exconevw}? Can it possibly hold for $\dim(V)\ge2$?

3.\
Given $G$ with $G(R)$ \sa ally compact, consider linear
representation spaces $V$ of $G$ (or more general affine
$G$-varieties). Under what conditions on $V$ is $S_0=\bigl(\Sigma
R[V]^2\bigr)^G$ finitely generated as a preordering in $R[V]^G$?
See Example \ref{s0notfg}.

4.\
Assume that $K\subset V(R)$ is $G$-invariant and that the saturated
preordering $\scrP_V(K)$ is finitely generated. Does this imply that
$\scrP_V(K)^G=\scrP_W(\pi K)$ is finitely generated as well (as a
preordering in $R[W]=R[V]^G$)?

5.\
How can Example \ref{salmabsp} be generalized, what is its essential
feature? Note that for such an example in $V=\A^n$, $n\ge2$, the
failure of the moment property for $K$ in $V$ cannot be due to the
containment of a cone in $K$, by the negative result \ref{exconevw}.

6.\
(C.f.\ Remark \ref{mpsolvbyinvfcts})
Let $K$ be a $G$-invariant basic closed set in $V(\R)$, and let $T$
be the preordering generated in $\R[V]$ by $\scrP_V(K)^G=\scrP_W
(\pi K)$. Then
$$T\subset\scrP_V(K)=\Sat_V(T).$$
Note that $T=\scrP_V(K)$ means that every $f\in R[V]$ with $f|_K\ge0$
can be written $f=\sum_ia_i^2f_i$ with $a_i$, $f_i\in R[V]$, $f_i|_K
\ge0$ and $f_i$ $G$-invariant. This motivates the following question:
Find necessary and sufficient conditions so that $T=\scrP_V(K)$, or
so that $T$ is dense in $\scrP_V(K)$. Note that, in general, $T$
fails to be dense in $\scrP_V(K)$ (Remark \ref{mpsolvbyinvfcts}).

\oldversion{%
\begin{lab}
Regarding Question 4.\ above, let us observe that
$T=\scrP_V(K)$ (or
$T$ is dense in $\scrP_V(K)$) need not hold in general:

Let $K$ be a $G$-invariant basic closed subset of $V(R)$,
$K:=\scrS_V(f_1,\dots,f_r)$ with $f_1,\dots,f_r\in R[V]^G$, and let
$M:=PO_V(f_1,\dots,f_r)$ the corresponding finitely generated
$G$-invariant preordering in $R[V]$. Note that $M^G$ is the
 $S_0$-module generated by $f_1,\dots,f_r$
in $R[V]^G$ (Proposition \ref{lem2po}).
 Assume that
$M^G$ is saturated
(as a preordering of $R[V]^G$). Let $T$ be as in Question 4.\ above;
then
$$T=\{f\in R[V] ; f=\sum_ia_i^2h_i , a_i, h_i \in R[V], h_i \in M^G =
\scrP_V(K)^G\}.$$
Thus $T=M$.

If moreover $M$ is {\it not} dense in its saturation,
 $T=\scrP_V(K)$, or
$T$ is dense in $\scrP_V(K)$ cannot hold.

Examples of finitely generated preorderings $M$ which are not dense in
their saturation, but for which
$M^G$ is saturated, do exist, see \cite{CKM}.
\end{lab}
}



\begin{thebibliography}{PoSch}

\bibitem[Ba]{Ba}
H.~Bauer:
\emph{Ma\ss- und Integrationstheorie}. 2.~Aufl.
De Gruyter, Berlin, 1992.


\bibitem[Bi]{Bi}
D.~Birkes:
Orbits of linear algebraic groups.
Ann.\ Math.\ (2) \textbf{93}, 459--475 (1971).

\bibitem[BCR]{BCR}
J.~Bochnak, M.~Coste, M.-F.~Roy:
\emph{Real Algebraic Geometry}.
Erg.\ Math.\ Grenzgeb.\ (3) \textbf{36}, Springer, Berlin, 1998.

\bibitem[Bo]{Bo}
A.~Borel:
\emph{Linear Algebraic Groups}. Second enlarged edition.
Grad.\ Texts Math.\ \textbf{126}, Springer, New York, 1991.

\bibitem[Br]{Br}
L.~Br\"ocker:
On symmetric semialgebraic sets and orbit spaces.
In: Singularities Symposium --- \L ojasiewicz 70 (Krak\'ow, 1996).
Banach Center Publ.\ \textbf{44}, Polish Acad.\ Sci., Warszawa, 1998,
pp.\ 37--50.

\bibitem[CKM]{CKM}
J.~Cimpri\v c, S.~Kuhlmann, M.~Marshall:
Positivity and $G$-invariant sums of squares.
Preprint, 2006.

\bibitem[DK]{DK}
H.~Derksen, G.~Kemper:
\emph{Computational Invariant Theory}.
Encycl.\ Math.\ Sciences, Springer, Berlin, 2002.

\bibitem[Ga]{Ga}
F.~R. Gantmacher:
\emph{Matrix Theory}.
Chelsea Publishing, New York, 1960.

\bibitem[GP]{GP}
K. Gatermann, P.~A. Parrilo:
Symmetry groups, semidefinite programs, and sums of squares.
J.~Pure Appl.\ Algebra \textbf{192}, 95--128 (2004).

\bibitem[HP]{HP}
J. W.~Helton, M.~Putinar:
Positive polynomials in scalar and matrix variables, the spectral
theorem and optimization.
Preprint, 2006.

\bibitem[KS]{KS}
M.~Knebusch, C.~Scheiderer:
\emph{Einf\"uhrung in die reelle Algebra}.
Vieweg, Braunschweig/ Wiesbaden, 1989.

\bibitem[KH]{KH}
B.~Kostant, G.~Hochschild:
Differential forms and Lie algebra cohomology for algebraic linear
groups.
Illinois J. Math.\ \textbf{6}, 264--281 (1962).

\bibitem[Kr]{Kr}
H.~Kraft:
\emph{Geometrische Methoden in der Invariantentheorie}.
Aspects of Mathematics, Vieweg, Braunschweig/Wiesbaden, 1984.

\bibitem[KM]{KM}
S.~Kuhlmann, M.~Marshall:
Positivity, sums of squares and the multi-dimensional moment problem.
Trans.\ Am.\ Math.\ Soc.\ \textbf{354}, 4285--4301 (2002).

\bibitem[KMS]{KMS}
S.~Kuhlmann, M.~Marshall, N.~Schwartz:
Positivity, sums of squares and the multi-dimensional moment
problem~II.
Adv.\ Geom.\ \textbf{5}, 583--607 (2005).

\bibitem[Lu]{Lu}
D.~Luna:
Sur certains op\'erations diff\'erentiables des groupes de Lie.
Am.\ J. Math.\ \textbf{97}, 172--181 (1975).

\bibitem[PSch]{PSch}
V.~Powers, C.~Scheiderer:
The moment problem for non-compact semialgebraic sets.
Adv.\ Geom.\ \textbf{1}, 71--88 (2001).

\bibitem[PW]{PW}
V.~Powers, Th.~W\"ormann:
An algorithm for sums of squares of real polynomials.
J.~Pure Appl.\ Algebra \textbf{127}, 99--104 (1998).

\bibitem[PD]{PD}
A.~Prestel, Ch.~N.~Delzell:
\emph{Positive Polynomials}.
Monographs in Mathematics, Springer, Berlin, 2001.

\bibitem[PS]{PS}
C.~Procesi, G.~Schwarz:
Inequalities defining orbit spaces.
Invent.\ math.\ \textbf{81}, 539--554 (1985).

\bibitem[Sch1]{Sch:MZ}
C.~Scheiderer:
Sums of squares on real algebraic curves.
Math.~Z. \textbf{245}, 725--760 (2003).

\bibitem[Sch2]{Sch:guide}
C.~Scheiderer:
Positivity and sums of squares: A guide to some recent results.
Preprint 2003, available at \texttt{www.ihp-raag.org/publications}.

\bibitem[Sch3]{Sch:stable}
C.~Scheiderer:
Non-existence of degree bounds for weighted sums of squares
representations.
J.~Complexity \textbf{21}, 823--844 (2005).

\bibitem[Sch4]{Sch:surf}
C.~Scheiderer:
Sums of squares on real algebraic surfaces.
Manuscr.\ math.\ \textbf{119}, 395--410 (2006).

\bibitem[Sch5]{Sch:localpo}
C.~Scheiderer:
Local study of two-dimensional preorderings.
In progress.

\bibitem[Sm]{Sm}
K.~Schm\"udgen:
On the moment problem of closed semi-algebraic sets.
J.~reine angew.\ Math.\ \textbf{558}, 225--234 (2003).

\bibitem[St]{St}
B.~Sturmfels:
\emph{Algorithms in Invariant Theory}.
Springer, Wien New York, 1993.
\end{thebibliography}
\end{document}